\documentclass[a4paper,11pt,reqno]{amsart}
\usepackage{amsmath}
\usepackage{amsfonts}
\usepackage{accents}
\usepackage{bbm}
\usepackage{amssymb}
\usepackage{graphicx}
\usepackage{geometry}
\usepackage{wrapfig}
\usepackage{ textcomp }
\usepackage{tikz-cd}[2010/10/13] 
\usepackage[small,nohug,heads=vee]{diagrams} 
\theoremstyle{plain}
\newtheorem{theorem}{Theorem}
\numberwithin{theorem}{section}

\newtheorem{conjecture}[theorem]{Conjecture}
\newtheorem{corollary}[theorem]{Corollary}

\newtheorem{definition}[theorem]{Definition}
\newtheorem{example}[theorem]{Example}

\newtheorem{lemma}[theorem]{Lemma}

\newtheorem{proposition}[theorem]{Proposition}

\newtheorem*{theoremX}{Theorem}

\newtheorem*{lemmaX}{Lemma}

\newcommand\myover[2]{\genfrac{}{}{0pt}{}{#1}{#2}}

\renewcommand{\H}{\mathrm{H}}
\newcommand{\DG}{{DG}}
\renewcommand{\mod}{\mathrm{mod}}
\newcommand{\comod}{\mathrm{comod}}
\newcommand{\C}{\mathrm{C}}
\newcommand{\Aut}{\mathrm{Aut}}
\newcommand{\Equiv}{\mathrm{Equiv}}
\newcommand{\Rep}{\mathrm{Rep}}
\newcommand{\Bigal}{\mathrm{Bigal}}

\newcommand{\Gal}{\mathrm{Gal}}
\newcommand{\End}{\mathrm{End}}

\newcommand{\Out}{\mathrm{Out}}
\newcommand{\Inn}{\mathrm{Inn}}
\newcommand{\Reg}{\mathrm{Reg}}
\newcommand{\Hom}{\mathrm{Hom}}
\newcommand{\id}{\mathrm{id}}
\newcommand{\im}{\mathrm{im}}

\renewcommand{\SS}{\mathbb{S}}

\newcommand{\F}{\mathbb{F}}

\newcommand{\Z}{\mathrm{Z}}

\renewcommand{\P}{\mathrm{P}}
\newcommand{\GL}{\mathrm{GL}}

\renewcommand{\d}{\mathrm{d}}
\newcommand{\ZZ}{\mathbb{Z}}

\newcommand{\nat}{\mathbb{N}}

\newcommand{\md}{\text{-}}

\usepackage{keyval}
\makeatletter
\define@key{setpar}{left}[0pt]{\leftmargin=#1}
\define@key{setpar}{right}[0pt]{\rightmargin=#1}
\define@key{setpar}{both}{\leftmargin=#1\relax\rightmargin=#1}
\makeatother

\newcommand{\hamburger}[4] 
{
  \thispagestyle{empty}
  \vspace*{-2cm}
  \begin{flushright}
    ZMP-HH #2 \\
    Hamburger Beitr\"age zur Mathematik Nr. #3 \\
    #4 \\
  \end{flushright}
  \vspace{0.5cm}
  \begin{center}
    \Large \bf
    #1
  \end{center}
  \vspace{0.5cm}
  \begin{center}        
    Simon Lentner, Jan Priel \\
    Fachbereich Mathematik, Universit\"at Hamburg \\
    Bereich Algebra und Zahlentheorie \\
    Bundesstra\ss e 55, D-20146 Hamburg \\
  \end{center}
  \vspace{0.5cm}
}
\begin{document}

\hamburger{On monoidal autoequivalences of the category of Yetter-Drinfeld modules over a group: The lazy case}{15-26}{572}{Nov. 2015}
\thispagestyle{empty}
\enlargethispage{1cm}

\begin{abstract}
An interesting open question is to determine the group of monoidal autoequivalences of the category of Yetter-Drinfeld modules over a finite group $G$, or equivalently the group of Bigalois objects over the dual of the Drinfeld double $\DG$. In particular one would hope to decompose this group into terms related to monoidal autoequivalences for the group algebra, the dual group algebra and interaction terms.

We report on our progress in this question: We first prove a decomposition of the group of Hopf algebra automorphisms of the Drinfeld double into three subgroups, which reduces in the case $G=\ZZ_p^n$ to a  Bruhat decomposition of $\GL_{2n}(\ZZ_p)$. Secondly, we propose a K\"unneth-like formula for the Hopf algebra cohomology of $\DG^*$ into three terms and prove partial results in the case of lazy cohomology. 
We use these results for the calculation of the Brauer-Picard group in the lazy case in \cite{LP15}. 
\end{abstract}

\makeatletter
\@setabstract
\makeatother

\tableofcontents
\newpage

\section{Introduction}

Given any monoidal category such as the representation category $H\md\mod$ of a finite dimensional Hopf algebra $H$, it is an interesting  and nontrivial question to calculate the group of monoidal autoequivalences of $H\md\mod$. In general, these are in bijection with so-called Bigalois objects \cite{Schau} of the dual Hopf algebra $H^*$, which are $H$-$H$-bimodule algebras satisfying a nondegeneracy condition. The autoequivalence is given by tensoring with the bimodule, while the algebra structure determines a monoidal structure on this functor.\\

Now given a Hopf algebra $H$ one can construct another Hopf algebra $DH$ called the Drinfeld double of $H$; modules over $DH$ are the Yetter-Drinfeld modules over $H$. We would like to calculate the monoidal autoequivalences of $DH\md\mod$ in terms of the monoidal autoequivalences of $H\md\mod$, $H^*\md\mod$ and interaction terms. In this article we address the case $H=kG$ the group algebra of a finite group $G$ and restrict ourselves to so-called lazy $2$-cocycles. We do not provide a complete decomposition, but we achieve partial results, that are however sufficient for our calculation of the corresponding subgroup of the Brauer-Picard group of $kG\md\mod$ in \cite{LP15}.\\  

In Section 2 we give some preliminaries. First we give a brief introduction to Hopf Bigalois objects and explain the notion of lazy Bigalois objects. These are given by pairs $(\phi,\sigma)$ where $\phi \in \Aut_{Hopf}(H\md\mod)$ describes the action of $A \otimes_H \bullet$ on the set of objects in $H\md\mod$ and where $\sigma \in \Z^2_L(H^*)$ is a so called lazy $2$-cocycle that describes the monoidal structure on the functor $A \otimes_H \bullet$. A similar description holds for arbitrary monoidal autoequivalences, but the Hopf $2$-cocycles do not form a group anymore.\\

Section \ref{seckupg} gives an explicit formula for lazy $2$-cocycles of $k^G$ using the classification of Galois algebras in \cite{Mov93} and \cite{Dav01}.\\

Section \ref{sec_cell} is devoted to the Hopf algebra automorphisms $\Aut_{Hopf}(\DG)$ and contains our main result. We give a (double) coset decomposition of $\Aut_{Hopf}(\DG)$ with respect to certain natural subgroups. For this we use the classification of bicrossed products of Hopf algebras and their homomorphisms \cite{ABM} in the special case $\DG = k^G \rtimes kG$ \cite{Keil}. In this approach, Hopf automorphisms of $\DG$ are described in terms of $2\times 2$-matrices with entries determined by certain group homomorphisms (see Proposition \ref{propone}). In the case that $G$ has no abelian direct factors, Keilberg determined completely $\Aut_{Hopf}(\DG)$ in terms of an exact factorization of $\Aut_{Hopf}(\DG)$ into natural subgroups, see Proposition \ref{auto} - \ref{autc} and Theorem \ref{thm_keilberg}. These subgroups are the upper triangular matrices $E$, lower triangular matrices $B$, two groups $V \cong \Aut(G)$ and $V_c \cong \Aut_c(G)$ that are certain diagonal matrices. Schauenburg and Keilberg have determined $\Aut_{Hopf}(\DG)$ in \cite{KS14} for general $G$ with a different approach, that we did not find suitable for analyzing the subgroup of \emph{braided} monoidal autoequivalences, which is our main motivation for the decomposition.

Our observation for the structure of $\Aut_{Hopf}(\DG)$ is that when $G$ contains a direct abelian factor $C$ we need a new class of Hopf automorphisms $R$ of $\DG$ called \emph{reflections} of $C$. These automorphisms exchange a direct factor $C$ with its dual $\hat{C}$. This leads us to the main result of Section \ref{sec_cell}: \\

\begin{theoremX}[\ref{thm_cell}]~\\
With the subgroups $B,E,V,V_c$ and subsets $R,R_t$ of $\Aut_{Hopf}(\DG)$ defined above the following holds:
\begin{itemize}
\item[(ii)] For every $\phi \in \Aut_{Hopf}(\DG)$ there is a twisted reflection $r\in R_t$ such that $\phi$ is an element in the double coset  
\begin{equation*} \begin{split}
&[(V_c \rtimes V) \ltimes B] \cdot r \cdot [(V_c \rtimes V) \ltimes E]
\end{split}
\end{equation*}
\item[(v)] For every $\phi \in \Aut_{Hopf}(\DG)$ there is a reflection $r = r_{(C,H,\delta)} \in R$ such that $\phi$ is an element in 
\begin{equation*} 
\begin{split}
& ((V_c \rtimes V) \ltimes B)E \cdot r 
\end{split}
\end{equation*}
\end{itemize}
\end{theoremX}

\noindent
In the case of an additive group $G=\F_p^n$ (Example \ref{exm_Fp_Aut}) all twisted reflections are proper reflections and we obtain the Bruhat decomposition of $\Aut_{Hopf}(\DG)=\GL_{2n}(\F_p)$, which is a Lie group of type $A_{2n-1}$, relative to the parabolic subsystem $A_{n-1}\times A_{n-1}$. More precisely, the Levi subgroup is $V_c \rtimes V=\GL_n(\F_p) \times \GL_n(\F_p)$ and $E\cong B$ are the solvable part of the parabolic subgroup. Representatives of the reflections are representatives of the $n+1$ double cosets of the Weyl group $\SS_{2n}$ with respect to the parabolic Weyl group $\SS_n\times \SS_n$. In more general cases usually $E\ncong B$.\\

In Section \ref{seclazy} we address the problem of decomposing the Hopf cohomology of the Drinfeld double $\H^2(DG^*)$ and in particular the group of lazy $2$-cocycles  $\H^2_L(DG^*)$. For Hopf algebra tensor products such as $kG\otimes k^G$ and Doi twists of these, the Kac-Schauenburg sequence \cite{Schau02} implies among others
\begin{align*}
\H^2(kG\otimes k^G)&\simeq \H^2(G,k^\times)\times \P(kG,k^G)\times \H^2(k^G)\quad\mbox{as sets}\\
\H^2_L(kG\otimes k^G)& \simeq \H^2(G,k^\times)\times \P_L(kG,k^G)\times \H^2_L(k^G) \quad\mbox{as groups}
\end{align*}
which should be directly compared to the K\"unneth formula for topological spaces 
$$\H^2(X\times Y)\simeq \H^2(X)\times \left(\H^1(X)\otimes \H^1(Y)\right)\times \H^2(Y)
\hspace{1.7cm}$$
However, for the dual Drinfeld double $\DG^*$, which is a Drinfeld twist, therefore has a modified coalgebra structure, such an easy formula might not be true. Nevertheless, we would still hope that one can define suitable subgroups and prove an exact factorization of $\H^2_L(DG^*)$ into these subgroups and that a similar approach decomposes $\H^2(\DG^*)$. The following diagram gives a sketch of the idea:

%

\begin{center}
 \includegraphics[scale=1]{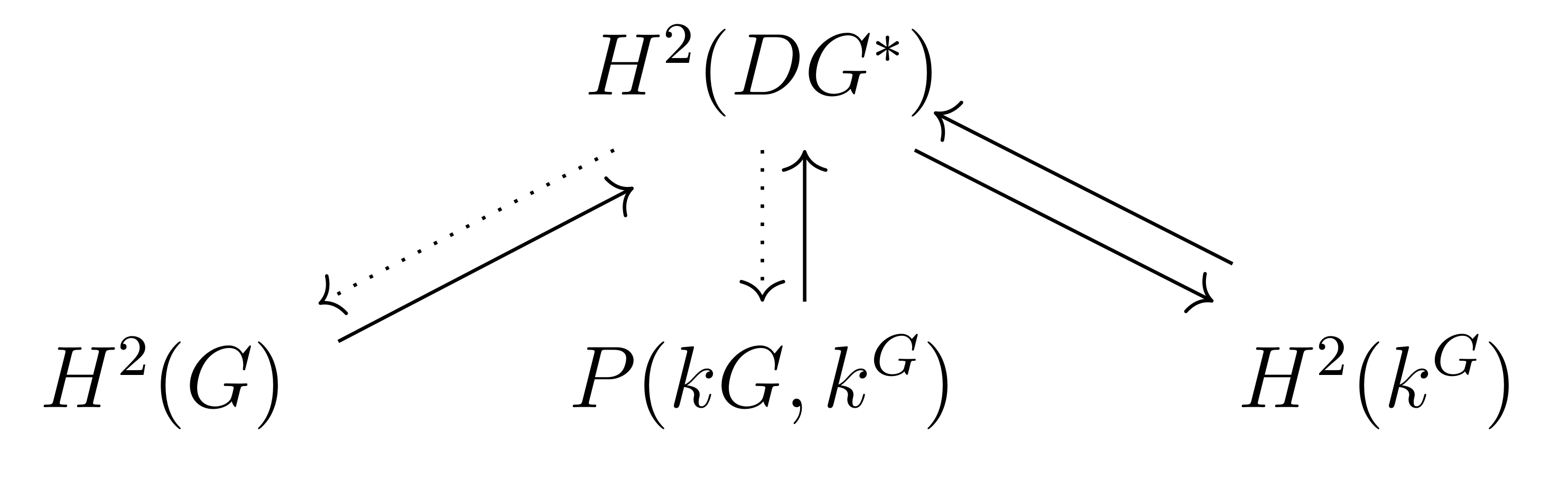}
\end{center}

Most of the arrows follow easily from functoriality. The dashed arrows indicate the direction we were not able to achieve in this article. Going from bottom-to-top gives us explicit subgroups (subsets) of $\H^2(DG^*)$ and we would hope for a decomposition. The maps top-to-bottom that would be necessary to prove this should follow from a restriction of the cocycle, however since $kG$ is not a Hopf subalgebra of $DG^*$ not all restrictions are well-defined. 

For \emph{lazy} $2$-cocycles we are able to solve some of these problems. On the other hand, here we have to deal with the fact that  lazy cohomology is {\it not functorial} which forces us to restrict the bottom-to-top maps on subgroups. One needs to prove that the images of the top-to-bottom maps are contained in these subgroups. Moreover, the lazy cohomology group $\H^2_L(DG^*)$ is not a mere subset of $\H^2(DG^*)$ but also a quotient by fewer coboundaries. In formulating the results, this is a tedious but not serious obstruction. \\

\noindent
We collect all maps established in Section \ref{seclazy} for the lazy cohomology:\\

\begin{center}
 \includegraphics[scale=0.9]{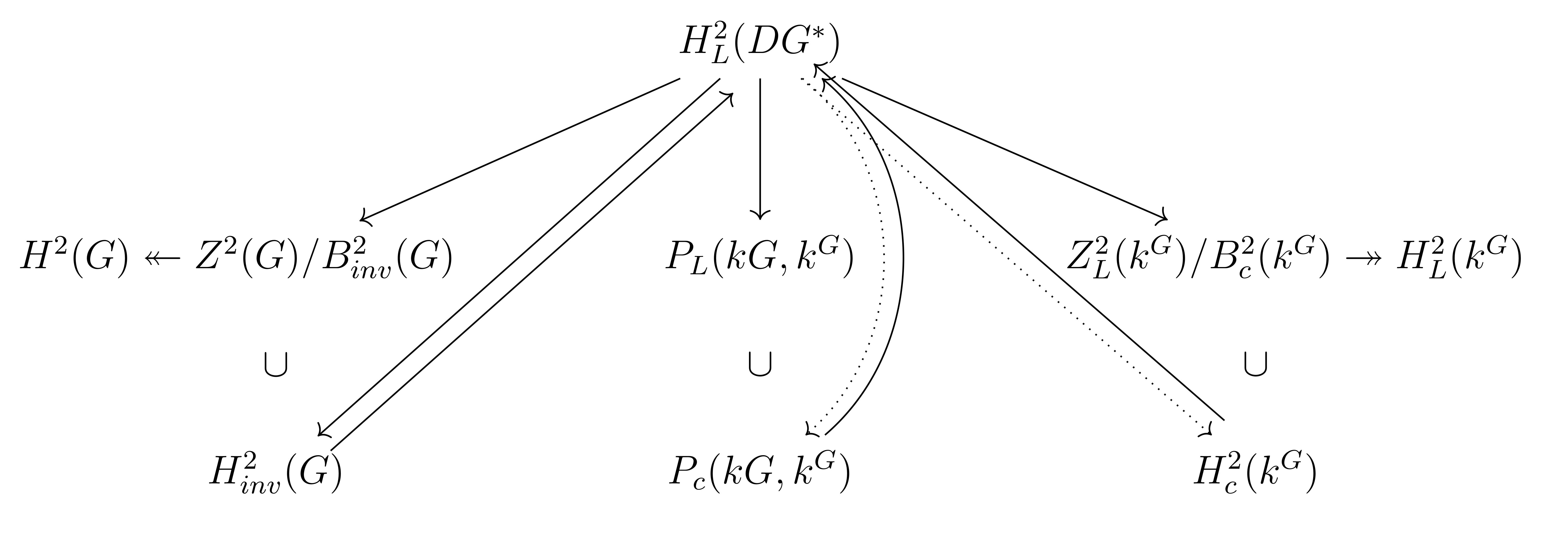}
\end{center}~\\


\noindent

These maps together with the following Lemmas are partial result that are needed to provide the full decomposition and are in addition necessary for our application in \cite{LP15}. 

\begin{lemmaX}\ref{lm_inKernelExact}
A lazy $2$-cocycle in the kernel of all three restriction maps (top-to-bottom) is already cohomologically trivial in $\H^2(\DG^*)$ but not necessarily in $\H^2_L(\DG^*)$. \\
\end{lemmaX}

\begin{lemmaX}\ref{lm_symmetric}
A \emph{symmetric} lazy $2$-cocycle on $\DG^*$ is cohomologically equivalent in $\H^2(\DG^*)$ to a lazy $2$-cocycle in the image of $\Z^2_{inv}(G,k^\times) \to \Z^2_L(\DG^*)$. \\
\end{lemmaX}

\section{Preliminaries}

We assume the field $k$ is algebraically closed and has characteristic zero and we denote by $\hat{G}$ the group of $1$-dimensional characters of $G$. 

\noindent
We use some familiarity with standard notions and properties about Hopf algebras and representation theory that can be found in e.g. \cite{Kass94}. We will use the following notation for conjugation $g^t= t^{-1}gt$ whenever it is convenient. Also, we use the Sweedler notation for the coproduct of a Hopf algebra $H$: $\Delta(a)=a_1 \otimes a_2$ for $a \in H$. \\ Recall that the Drinfeld double $\DG$ of the Hopf algebra $kG$ for a finite group $G$ is a Hopf algebra defined as follows: As a vector space $\DG$ is the tensor product $k^G \otimes kG$ where $k^G:=(kG)^*=\Hom(kG,k)$. Denote the basis of $\DG$ by $\{e_x \times y\}_{x,y \in G}$ where $e_x \in k^G$ is defined by $e_x(y)=\delta_{x,y}$ for all $x,y \in G$. $\DG$ has the following Hopf algebra structure:  
$$(e_x \times y)(e_{x'} \times y') = e_x(yx'y^{-1}) (e_x \times yy') \qquad \Delta(e_x \times y) = \sum_{x_1x_2=x}(e_{x_1} \times y) \otimes (e_{x_2} \times y)$$
with the unit $1_{\DG} = \sum_{x \in G} (e_x \times 1_G)$, counit $\epsilon(e_x \times y) = \delta_{x,1_G}$ and antipode $S(e_x \times y) = e_{y^{-1}x^{-1}y} \times y^{-1}$. It is a well known fact that $\DG$ is the Doi-twist of the tensor product $k^G \otimes kG$. 

\noindent
We also use the Hopf algebra $\DG^*$ which is the dual Hopf algebra of $\DG$. For this we recall that in $\DG^*$ we have $$(x \times e_y)({x'} \times e_{y'}) = (xx' \times e_y*e_{y'}) \qquad \Delta(x \times e_y) = \sum_{y_1y_2=y}(x \times e_{y_1}) \otimes (y_1^{-1}xy_1 \times e_{y_2})$$ 
which are just the dual maps of the coproduct resp. the product in $\DG$. In the case the group $G=A$ is abelian $DA \simeq k(\hat{A}\times A)$ and $DA^* \simeq k(A\times\hat{A})$ are isomorphic to each other as Hopf algebras. In general there is no Hopf isomorphism from $\DG$ to $\DG^*$. $\DG \md\mod$ is equivalent as a (braided) monoidal category to the category of $kG$-Yetter-Drinfeld-modules and the Drinfeld center of the category of $G$-graded vector spaces. \\

\begin{definition}~
We denote by $\underline{\Aut}_{mon}(\DG\md \mod)$ the functor category of monoidal autoequivalences of $\DG \md \mod$ and natural monoidal isomorphisms. We will occasionally abuse notation by denoting only the objects of this category by $\underline{\Aut}_{mon}(\DG\md \mod)$. Let then ${\Aut_{mon}}(\DG\md \mod)$ be the group of isomorphism classes of monoidal autoequivalences of $\DG \md \mod$. \\
\end{definition}

\medskip

\subsection{Hopf-Galois-Extensions}~\\

\noindent
In order to study monoidal automorphisms of $\DG\md\mod$ we will make use of the theory of Hopf-Galois extensions. For this our main source is \cite{Schau} and \cite{Schau2}. The motivation for this approach lies mainly in the relationship between Galois extensions and monoidal functors as formulated e.g. in \cite{Schau} and also stated in Proposition \ref{fib}. Namely, monoidal functors between the category of $H$-comodules and the category of $L$-comodules are in one-to-one correspondence with $L$-$H$-Bigalois objects. For this reason we are lead to the study of $DG^*$-Bigalois extensions. Since $\DG$ is finite dimensional we can use the fact that Bigalois objects over a finite dimensional Hopf algebra can essentially be described by an automorphism of $H$ and a $2$-cocycle on $H$. We will see in a later section that it is possible to handle the automorphism group of $\DG$, but the quite large set of cocycles. There is however a special class of $2$-cocycles that are called lazy (sometimes invariant) which have a better behavior in a certain sense. These give us a large class of Bigalois objects. First, we would like to introduce some basic notions and properties of Hopf-Galois extensions.

\medskip

\begin{definition} Let $H$ be a bialgebra over $k$. A $k$-algebra $A$ is a right $H$-comodule algebra if it has a right $H$-comultiplication $\delta_R:A \to A \otimes_k H$ such that $\delta_R$ is an algebra map. A right $H$-comodule algebra $A$ is called a \emph{right} $H$-\emph{Galois extension of} $B:=A^{coH}$ if $A$ is faithfully flat over $B$ and the Galois map 
\begin{diagram} 
&A \otimes_{B} A &\rTo^{id_A \otimes \delta} &A \otimes A \otimes H &\rTo^{\mu_A \otimes id_H} &A \otimes H \\
&x \otimes y &\rMapsto^{} &x \otimes y_0 \otimes y_1 &\rMapsto^{} &xy_0 \otimes y_1 
\end{diagram} 
is a bijection. Here $A^{coH} := \{ a \in A ~|~ \delta_R(a)= a \otimes 1_H \}$ are the coinvariants of $H$ on $A$. A morphism of right $H$-Galois objects is an $H$-colinear algebra morphism. Left $H$-Galois extensions are defined similarly. Denote by $ \mathrm{Gal}_B(H)$ the set of equivalence classes of right $H$-Galois extensions of $B$. A right $H$-\emph{Galois object} is a Galois extension of the base field $B=A^{coH}=k$. 
\end{definition}

\medskip
\newpage
\begin{example}~
\begin{itemize}
\item A Hopf algebra $H$ has a natural structure of an $H$-Galois object where the coaction is given by the comultiplication in $H$. \\ 
Note that a bialgebra $H$ is a $H$-Galois object if and only if $H$ is a Hopf algebra. 
\item Let $A/k$ be a Galois field extension of $k$ with finite Galois group $G:=\Aut(A/k)$ which is the group of $k$-linear field automorphisms. Setting $H:=k^G$ the $k$-algebra $A$ is then an $H$-Galois extension of $k$ with the obvious comodule structure.  
\end{itemize}
\end{example}

\medskip

\begin{definition} Let $L,H$ be two Hopf algebras. An $L$-$H$-\emph{Bigalois object} $A$ is an $L$-$H$-bicomodule algebra which is a left $H$-Galois object and a right $L$-Galois object. Denote by $\underline{\Bigal}(L,H)$ the set of $L$-$H$-Bigalois objects and by $\underline{\Bigal}(H)$ the set of $H$-$H$-Bigalois objects. Denote by $\Bigal(L,H)$ the set of isomorphism classes of $L$-$H$-Bigalois objects and by $\Bigal(H)$ the set of isomorphism classes of $H$-$H$-Bigalois objects.    
\end{definition}

\medskip

\noindent
Recall that the \emph{cotensor product} of a right $L$-comodule $(A,\delta_R)$ and a left $L$-comodule $(B,\delta_L)$ is defined by
$$A\square_L B:=\{a\otimes b\in A\otimes B\mid \delta_R(a)\otimes b=a\otimes \delta_L(b)\}$$
Moreover, if $A$ is an $E$-$L$-Bigalois object and $B$ an $L$-$H$-Bigalois object then the cotensor product $A\square_L B$ is an $E$-$H$-Bigalois object. 

\medskip

\begin{proposition} The cotensor product gives $\Bigal(H)$ a group structure. The Hopf algebra $H$ with the natural $H$-$H$-Bigalois object structure is the unit in the group $\Bigal(H)$. Further, we can define $\Bigal$ to be a groupoid where the objects are given by Hopf algebras and the morphisms between two Hopf algebras $L$, $H$ are given by elements in $\Bigal(L,H)$. The composition of morphisms is the cotensor product.  
\label{groupoid}
\end{proposition}

\medskip

\begin{proposition} Given an right $H$-Galois object $A$ there exists a Hopf algebra $L(H,A)$ such that $A$ is an $L(H,A)$-$H$-Bigalois object. If $L'$ is another Hopf algebra such that $A$ is an $L'$-$H$-Bigalois object then there is an isomorphism $L(H,A) \cong L'$ that is compatible with the respective coactions on $A$. In particular, for a given right $H$-Galois object $A$ the set of all 
$L(H,A)$-$H$-Bigalois object structures on $A$ is parametrized by $\Aut_{Hopf}(L(H,A))$ up to coinner Hopf automorphisms.
\label{TK}
\end{proposition}

\medskip

\begin{example}
For a Hopf algebra automorphism $\phi\in \Aut_{Hopf}(H)$ we obtain an $H$-$H$-Bigalois object ${_\phi}H$ where the left $H$-coaction is the coproduct post-composed with $\phi$. This yields a group homomorphism $\Aut_{Hopf}(H)\to \Bigal(H)$ which in general is neither surjective nor injective.
\label{exp1}
\end{example}

\noindent
Let us now discuss the relation between Galois extensions and monoidal functors, which we will later use to study the group $\Aut_{mon}(\DG \md \mod)$. For this recall that given a Hopf algebra $H$ a \emph{fiber functor} $H \md \mathrm{comod} \rightarrow \mathrm{Vect}_k$ is a $k$-linear, monoidal, exact and faithful functor that preserves colimits. We denote by $\mathrm{Fun}_{fib}(H\md\mathrm{ comod},\mathrm{Vect}_k)$ the set of monoidal isomorphism classes of fiber functors.   

\begin{proposition}(\cite{Schau} Sect. 5) \\ 
Let $H$, $L$ be Hopf algebras then the following holds: 
\begin{itemize}
\item There is a bijection of sets 
\begin{equation}
\begin{split}
\mathrm{Gal}(H) &\stackrel{\sim}{\rightarrow} \mathrm{Fun}_{fib}(H\md \mathrm{comod},\mathrm{Vect}_k) \\
A &\mapsto (A \square_H \bullet,J^A) 
\end{split}
\label{eqn:fibfun}
\end{equation}
\item There is an equivalence of groupoids:
\begin{align*}
H &\mapsto H\md\comod \\
\Bigal(L,H) &\simeq \Equiv_{mon}(H\md\mathrm{comod},L \md\mathrm{comod}) \\
A &\mapsto (A \square_H \bullet,J^A) 
\end{align*} 
\end{itemize}
where the monoidal structure $J^A$ of the functor $ A \square_H \bullet$ is given by 
\begin{equation}
\begin{split} 
J^A_{V,W}:(A \square_H V) \otimes_k (A \square_H W) &\stackrel{\sim}{\rightarrow} A \square_H (V \otimes_k W) \\ 
     \left(\sum x_i \otimes v_i\right)\otimes \left(\sum y_i \otimes w_i\right) &\mapsto \sum x_iy_i \otimes v_i \otimes w_i 
\end{split}
\label{eqn:monoidal1}
\end{equation}
In particular, we have $\Bigal(H^*) \simeq \Aut_{mon}(H^*\md\comod) \simeq \Aut_{mon}(H \md\mod)$. 
\label{fib}
\end{proposition}
\medskip

\begin{example}
In Example \ref{exp1} we obtained for each $\phi\in \Aut_{Hopf}(H) \cong \Aut_{Hopf}(H^*)$ a $H^*$-Bigalois object $_\phi H^*$ isomorphic to $H^*$ as an algebra but with right comodule structure post-composed by $\phi$. Under the isomorphism above, this corresponds to the monoidal autoequivalence $(F_\phi,J^{triv})$ mapping an $H$-module $M$ to the $H$-module $_\phi M$ given by pre-composing the module structure with $\phi$ (and a trivial monoidal structure). \\
\label{fphi}
\end{example}

There is a large class of $H$-Galois extensions which come from twisting the algebra structure. In the case when $H$ is finite-dimensional or pointed all $H$-Galois extensions arise in this way. From now on we are using Sweedler notation: $\Delta(h) = h_1 \otimes h_2$. \\

\begin{definition}~
\begin{itemize}
\item Denote by $\Reg^1(H)$ the group of convolution invertible, $k$-linear maps $\eta:H \rightarrow k$ such that $\eta(1) = 1$. 
\item  Let $\Reg^2(H)$ be the group of convolution invertible, $k$-linear maps $\sigma:H\otimes H \rightarrow k$ such that $\sigma(1,h)=\epsilon(h)=\sigma(h,1)$
\item A left $2$-\emph{cocycle} on $H$ is a map $\sigma \in \Reg^2(H)$ such that for all $a,b \in H$  
\begin{equation} 
\sigma(a_{1},b_{1})\sigma(a_{2}b_{2},c) = \sigma(b_{1},c_{1})\sigma(a,b_{2}c_{2}) 
\label{cocycle}
\end{equation}
We denote the set of $2$-cocycles on $H$ by $\Z^2(H)$. 
\item We define a map $\d:\Reg^1(H) \rightarrow \Reg^2(H)$ by
$\d\eta (a,b) = \eta(a_{1})\eta(b_{1})\eta^{-1}(a_{2}b_{2})$ for all $a,b \in H$. We have $\d\eta \in \Z^2(H)$ and denote by $\H^2(H)$ the set of $2$-cocycles modulo the image of $\d$. For other properties of $\d$ see \cite{BiCa} Lemma 1.6. 
\end{itemize}
\label{regs}
\end{definition}
    
\begin{proposition} 
An $H$-Galois object $A$ is called \emph{cleft} if one of the following equivalent conditions hold:
\begin{itemize} 
\item There exists an $H$-comodule isomorphism $H \simeq A$. 
\item There exists an $H$-colinear convolution invertible map $H \rightarrow A$ 
\item There exists a $2$-cocycle $\sigma$ such that $A \simeq{_\sigma}H$ as $H$-comodule algebras, where ${_\sigma}H$ is has the $H$-comodule structure of $H$ and the following twisted algebra structure $$a \cdot_\sigma b = \sigma(a_1,b_1)a_2b_2$$ 
\end{itemize}
Two sufficient conditions such that any $H$-Galois object $A$ is cleft are: 
\begin{itemize} 
\item $H$ is finite dimensional. 
\item $H$ is pointed (all simple comodules are $1$-dimensional).  
\end{itemize}
In this case the map $\sigma \mapsto {_\sigma}H$ induces an bijection $\H^2(H) \cong \Gal(H)$. \\ Also, if we are in the cleft case the unique Hopf algebra $L(H,A)$ from Proposition \ref{TK} is given by the Doi twist $L(H,A):={_\sigma}H_{\sigma^{-1}}$ which is $H$ as a coalgebra and has the following twisted algebra structure: 
$a \cdot b := \sigma(a_1,b_1)a_2b_2\sigma^{-1}(a_3,b_3) $ 
\label{prop_cleft}
\end{proposition}
It is important to note that $\Z^2(H)$ (as well as $\H^2(H)$) is \emph{not} a group, because the convolution of two $2$-cocycles is not a $2$-cocycle in general. The convolution product $\sigma * \tau$ for $\sigma \in \Z^2(H)$ and $\tau \in \Z^2(L)$ is a $2$-cocycle in $Z^2(H)$ if $L \simeq {_\sigma}H_{\sigma^{-1}}$. One should rather, analogous to Proposition \ref{groupoid}, consider the groupoid where the objects are Hopf algebras and the morphisms between two Hopf algebras $H,L$ are given by those $2$-cocycles $\sigma \in \Z^2(H)$ that have the property $L \simeq {_\sigma}H_{\sigma^{-1}}$. 

\medskip
\noindent
Moreover, if $F:H\md\mathrm{comod} \rightarrow \mathrm{Vect}_k$ is a fiber functor then $F$ corresponds under (\ref{eqn:fibfun}) to a cleft $H$-Galois object if and only if $F$ is isomorphic to the forgetful functor if and only if $F$ preserves the dimensions of the $k$-vector spaces underlying the $H$-comodules. 

\medskip

\begin{corollary}
For a finite-dimensional Hopf algebra $H$ we have a surjection of sets
\begin{align*}
\mathrm{Z}^2(H^*) &\rightarrow \mathrm{Fun}_{fib}(H\md\mod,\mathrm{Vect}_k)  &J^{\sigma}_{V,W}:V \otimes_k W \stackrel{\simeq}{\rightarrow} V \otimes_k W \\
\sigma &\mapsto (\mathrm{ Forget},J^\sigma)                                   &v \otimes w \mapsto \sigma_1.v \otimes \sigma_2.w  
\end{align*}
Here we have identified $\sigma:H^* \otimes H^* \rightarrow k$ with an element $\sigma = \sigma_1 \otimes \sigma_2 \in H \otimes H$. \newline 
Further, this map induces a bijection of sets $\H^2(H^*) \simeq \mathrm{Fun}_{fib}(H\md\mod,\mathrm{Vect}_k)$. 
\label{jsigma}
\end{corollary}
\begin{proof} Use Proposition \ref{fib} and the fact that since $H$ is finite dimensional all  $H$-Galois objects are cleft. Hence the corresponding functor is the forgetful functor on objects and morphisms. Further, it is a straightforward calculation that if use the canonical equivalence $\comod\md H^* \cong H\md \mod$ and view the $2$-cocycle as an element $\sigma \in H \otimes H$ that the monoidal structure as defined in (\ref{eqn:monoidal1}) is equivalent to acting with $\sigma$. Bijectivity follows from Proposition \ref{prop_cleft}  
\end{proof}

\medskip

\noindent
In general it is very hard to control the sets $\Z^2(H)$, $\H^2(H)$ and even more the subset of Galois objects with $L(A,H)\cong H$. For this reason we consider below $H$-Galois objects that have an additional property. These $H$-Galois objects can be characterized as having a canonical choice of a Bigalois object structure and lead ultimately to Lemma \ref{lm_AutMonLazy}. An implication of that property is that they can be described by certain cohomology groups.\\

\subsection{Lazy Bigalois Objects and Lazy Cohomology}~\\

\begin{definition}\label{def_lazy}~
\begin{itemize}
\item An $H$-$H$-Bigalois object $A$ is called \emph{bicleft} if and only if $A \cong H$ as $H$-bicomodules. The group $\Bigal_{bicleft}(H)$ of bicleft $H$-$H$-Bigalois objects is a normal subgroup of $\Bigal(H)$. 
\item A right $H$-Galois object $A$ is called \emph{lazy} if there exists a unique left $H$-Galois structure such that $A$ is bicleft. Hence it is a Galois object where there is a canonical $L(H,A) \cong H$.
\item An $H$-$H$-Bigalois object $A$ is called \emph{lazy} if it is lazy as a right $H$-Galois object. Denote the group of lazy $H$-$H$-Bigalois objects by $\Bigal_{lazy}(H)$.
\end{itemize}
\end{definition}

\begin{example}
If $H$ is cocommutative, then all $H$-Galois objects and $H$-$H$-Bigalois objects are lazy.
\end{example}

\medskip

Since we are here mainly interested in the cleft case, hence the case where all $H$-Galois objects are of the form ${_\sigma}H$ for a $2$-cocycle $\sigma$, we discuss what additional property on $\sigma$ corresponds to the lazy property of the Galois object ${_\sigma}H$.   

\medskip

\begin{definition}~
\begin{itemize}
\item An $\eta \in \Reg^1(H)$ is called \emph{lazy} if it has the additional property $\eta * \id = \id * \eta$. Denote by $\mathrm{Reg}^1_L(H)$ the subgroup of such lazy regular maps.      
\item A $\sigma \in \Reg^2(H)$ is called \emph{lazy} if it commutes with the multiplication on $H$:  $$ \sigma*\mu_H = \mu_H * \sigma $$ 
and denote the subgroup of such lazy regular maps by $\Reg_L^2(H)$. Accordingly, a $2$-cocycle $\sigma \in \Z^2(H)$ is called \emph{lazy} if $\sigma \in \Reg^2_L(H)$. Denote by $\Z^2_L(H)$ the subgroup of lazy $2$-cocycles. \\ Note: The map $\d$ in Definition \ref{regs} maps $\Reg^1_L(H)$ to the center of $\Z^2_L(H)$. 
\item An $\eta \in \Reg^1(H)$ is called \emph{almost lazy} if $\d \eta$ is lazy. Denote by $\Reg_{aL}^1(H)$ the group of such almost lazy regular maps.    
\item The lazy second cohomology group is then defined by
\begin{align*}
\H^2_L(H)&:= \Z_L^2(H)/\d(\Reg_L^1(H))
\end{align*}
The set of lazy cohomology classes is defined as the set of cosets
\begin{align*}
\H^2_{aL}(H) &:= \Z_L^2(H)/\d(\Reg_{aL}^1(H)) 
\end{align*}
\end{itemize}
\end{definition}

\medskip

\begin{proposition}(\cite{BiCa} Proposition 3.6, Proposition 3.7) \\ 
An $H$-Galois object $A$ is bicleft if and only if there exists a lazy $\sigma \in Z_L^2(H)$ such that ${_{\sigma}}H_{\sigma^{-1}} \simeq A$ as $H$-bicomodule algebras. Further, the group of bicleft $H$-Bigalois objects is a normal subgroup of $\Bigal(H)$ and $\H^2_L(H) \simeq \Bigal_{bicleft}(H)$.  
\end{proposition}

\medskip

\begin{example}~\label{exm_LazyIff} 
\begin{itemize}
\item For $\H=kG$ with $G$ a finite group the $2$-cocycle condition reduces to  
$\sigma(ab,c)\sigma(a,b) = \sigma(b,c)\sigma(a,bc)$. 
The lazy condition is automatically fulfilled and $\d$ is the differential corresponding to the bar complex, hence
$\H^2(kG) = \H^2_L(kG) = \H^2(G)$     
\item For $\H=k^G$ with $G$ a finite group, a left $2$-cocycle $\alpha$ is lazy if and only if for all $x,y,g \in G$ holds: 
$\alpha(e_{x},e_{y}) = \alpha(e_{gxg^{-1}},e_{gyg^{-1}})$. An $\eta \in \Reg^1(k^G)$ is lazy if and only if $\eta(e_x) = \eta(e_{gxg^{-1}})$ for all $g,x \in G$. \\
The non cocommutativity of $k^G$ makes a more explicit characterization quite hard. In fact, will see in Section \ref{seckupg} that one needs a lot of machinery to determine lazy $k^G$-Bigalois objects.
\end{itemize}
\end{example}

\medskip 

\begin{definition}~
\begin{itemize}
\item For a Hopf algebra $H$ denote by $\mathrm {Int}(H)$ the subgroup of \emph{internal} Hopf automorphisms. These are such $\phi \in \Aut_{Hopf}(H)$ that are of the form $\phi(h) = x h x^{-1}$ for some invertible $x \in \H$ and $x$ has the property that for all $h \in H$ \begin{equation} (x \otimes x)\Delta(x^{-1})\Delta(h) = \Delta(h)(x \otimes x)\Delta(x^{-1}) \label{eqn:int}\end{equation} 
\noindent
Note: For an invertible element $x \in H$ the conjugation $\phi(h) = x h x^{-1}$ is an algebra automorphism. It is a coalgebra automorphism if and only if $(\ref{eqn:int})$ holds.
\item Denote by $\mathrm {Inn}(H) \subset \mathrm{Int}(H)$ the subgroup of \emph{inner} Hopf automorphisms, hence $\phi \in \Aut_{Hopf}(H)$ of the form $\phi(h) = xhx^{-1}$ for some group-like $x \in H$.
\item Let $\mathrm{Out}(H) := \Aut_{Hopf}(H)/\mathrm{Inn}(H)$ the subgroup of outer Hopf automorphisms. 
\end{itemize}
\end{definition}

\begin{example}
For $H=\DG$ the group-like elements are $G(\DG)=\hat{G}\times G$, hence $\Inn(\DG) \cong \Inn(G)$. More precisely, each inner automorphism $\phi \in \Inn(\DG)$ is of the form $\phi(e_g \times h) = e_{tgt^{-1}}\times tgt^{-1}$ for some $t \in G$. 
\end{example}

\medskip

\noindent
We now discuss how the previously defined subgroups interact:

\medskip

\begin{lemma}~(\cite{BiCa} Lemma 1.15)\\
$\Aut_{Hopf}(H)$ acts on $Z^2_L(H^*)$ by $\phi.\sigma = (\phi \otimes \phi)(\sigma)$
where we identify a $2$-cocycle in $H^*$ with $\sigma = \sigma_{1} \otimes \sigma_2 \in H \otimes H$. Then for all $\phi \in \Aut_{Hopf}(H)$ we have: 
\begin{itemize} 
\item If $\omega,\sigma \in \mathrm {Z}^2(H^*)$ then $\phi.(\sigma*\omega) = (\phi.\sigma)*(\phi.\omega)$.  
\item If $\gamma \in \mathrm {Reg}^1(H^*)$ then $\phi.\d\gamma = \d(\gamma \circ \phi)$. 
\item $\Inn(H)$ acts trivially on $\Reg^2_L(H^*)$.
\item This action induces an action of $\mathrm {Out}(H)$ on lazy cohomology $\H^2_L(H^*)$ 
\end{itemize}
\label{action} 
\end{lemma}

\medskip

\begin{lemma} Let $\phi \in \Aut_{Hopf}(H)$ be a Hopf automorphism and let $\sigma \in \Z^2_L(H^*)$ a lazy $2$-cocycle then the following are equivalent
\begin{itemize}
\item The functor $(F_\phi,J^{\sigma})$ is monoidally equivalent to $(\id,J^{triv})$ 
\item $\phi=x \cdot \id_H \cdot x^{-1} \in \mathrm{Int}(H)$ for some invertible element $x \in H$ and 
\begin{equation} 
	\sigma = \Delta(x)(x^{-1} \otimes x^{-1}) 
\label{eqn:equa2}
\end{equation}
\end{itemize} 
\end{lemma}
\label{kernel1}
\begin{proof} 
Let $\eta$ be the monoidal equivalence $(\id,J^{triv}) \sim (F_{\phi},J^{\sigma}$). Then in particular there is an $H$-module isomorphism for the regular $H$-module $\eta_H: H \stackrel{\sim}{\rightarrow} F_{\phi}(H) =: {_\phi}H$ such that $\eta_H \circ f = f \circ \eta_H $ for all $H$-module homomorphisms $f:H \rightarrow H$. Note that $\eta_H$ is determined by what it does on $1_H$ hence by an invertible element $x := \eta(1) \in H$. Further, every $H$-module morphism $f:H \rightarrow H$ is determined by an $h := f(1) \in H$ and then the naturality property of $\eta$ implies $\phi(h) = x h x^{-1}$.  
Since $\phi$ is a Hopf automorphism there are additional conditions on $x$. It is obvious that $\phi$ is an algebra automorphism for all invertible $x \in H$ but it is a coalgebra automorphism if and only if 
\[ \Delta(x^{-1})\Delta(h)\Delta(x) = (x^{-1} \otimes x^{-1})\Delta(h)(x \otimes x)  \] which is equivalent to $(\ref{eqn:int})$. Further, by definition there has to be an $H$-module isomorphism $\eta_{H \otimes H}: H\otimes H \stackrel{\sim}{\rightarrow} {_\phi}(H \otimes H)$ such that for all $H$-module morphisms $r:H \rightarrow H \otimes H$ the following diagram commutes: 
\begin{diagram}
&H &\rTo^{\eta_H} &{_\phi}H \\
&\dTo^{r} & &\dTo^{r} \\ 
&H \otimes H &\rTo^{\eta_{H\otimes H}} &{_\phi}(H\otimes H)  
\end{diagram}
Again, $r$ is determined by $y:=r(1) \in H \otimes H$ then the diagram above implies $\eta_{H\otimes H}(y) = x.y$ for all $y \in H \otimes H$. Now we use that $\eta$ is monoidal, hence in particular the diagram 
\begin{diagram}
&H\otimes H       &\rTo^{\eta_H \otimes \eta_H} &{_\phi}H \otimes {_\phi}H \\
&\dTo_{J^{triv}_{H\otimes H}} &                             &\dTo_{J^{\sigma}_{H\otimes H}} \\ 
&H \otimes H     &\rTo^{\eta_{H\otimes H}}     &{_\phi}(H\otimes H)  
\end{diagram}
commutes which implies that $J^{\sigma}_{H\otimes H}(xh \otimes xh') = \Delta(x)(h \otimes h')$ for all $h,h' \in H$ which is equivalent to $(\ref{eqn:equa2})$. \\
On the other hand let $(\phi,\sigma) \in \Aut_{Hopf}(H)\times Z^2_L(H^*)$ such that $\phi(h) = x h x^{-1}$ for some invertible $x \in H$ and such that equations ($\ref{eqn:int}$) and ($\ref{eqn:equa2}$) hold. It can be checked by direct calculation that $\phi$ is a well-defined algebra automorphism because $x$ is invertible and coalgebra morphism because of ($\ref{eqn:int}$). We claim that the following family of morphisms $\eta = \{ \eta_M: M \rightarrow {_\phi}M;\eta_M(m) = x.m \}$ is a monoidal equivalence between $(\id,J^{triv})$ and $(F_{\phi},J^{\sigma})$. The fact that $\eta$ is a natural equivalence follows from construction. The fact that $\eta$ is monoidal follows from ($\ref{eqn:equa2}$). 
\end{proof}

We collect the statements of this section:

\begin{lemma}\label{lm_AutMonLazy}
There is an assignment of sets 
\begin{center}
\begin{tabular}{rcl}
$\Psi: \;\;\Aut_{Hopf}(H) \times \Z_L^2(H^*)$ & $\;\;\;\rightarrow$  $\underline{\Bigal}(H^*)$ & $\rightarrow\;\underline{\Aut}_{mon}(H \md \mod)$\\
$(\phi, \sigma)$ &$\mapsto$  ${_\phi}({_\sigma}H^*)$ & $\mapsto {_\phi}({_\sigma}H^*) \otimes_H \bullet \; =(F_\phi,J^\sigma)$
\end{tabular}
\end{center}
where ${_\phi}({_\sigma}H^*)$ is the cleft lazy $H^*$-$H^*$-Bigalois object defined by $H^*$ with multiplication deformed by $\sigma$, right $H^*$-coaction given by comultiplication and left $H^*$-coaction given by comultiplication post-composed with $\phi$.
We obtain a group homomorphism 
$$\Aut_{Hopf}(H) \ltimes \Z^2_L(H^*) \rightarrow \Aut_{mon}(H \md \mod)$$
where $\Aut_{Hopf}(H)$ acts on $\Z^2_L(H^*)$ as defined in  Proposition \ref{action}. This map factors into a group homomorphism 
\begin{equation} \Out_{Hopf}(H) \ltimes \H^2_L(H^*) \rightarrow \Aut_{mon}(H \md \mod) \label{outh2}\end{equation}
\end{lemma}
\begin{proof}
We first check that $(\ref{outh2})$ is indeed a homomorphism. The composition in the semi-direct product $(\phi,\sigma)(\phi',\sigma') = (\phi' \circ \phi, (\phi.\sigma')\sigma)$ is mapped to the functor $(F_{\phi' \circ \phi},J^{(\phi.\sigma')\sigma})$. On the other hand the composition $(F_\phi,J^{\sigma})$ with $(F_{\phi'},J^{\sigma'})$ gives $F_{\phi'}\circ F_{\phi}=F_{\phi' \circ \phi}$ with the monoidal structure $F_{\phi'}(J^{\sigma}_{M,N})\circ J^{\sigma'}_{M_{\phi},N_{\phi}}(m \otimes n) = \sigma.(\phi.\sigma').( m\otimes n) = J^{\sigma*(\phi.\sigma')}_{M,N}(m \otimes n)$\\
Let us now show that the map factorizes as indicated: The kernel of $\Aut_{Hopf}(H) \ltimes \Z^2_L(H^*) \to \Out_{Hopf}(H) \ltimes H^2_L(H^*)$ is given by the set of all $(\phi,\sigma)$ with $\phi(h)=tht^{-1}$ for some grouplike element $t\in G(H)$ and $\sigma=\d\eta$ for $\eta\in\Reg^1_L(H^*)$. To see that the functor $\Psi(\phi,\sigma)$ is in this case trivial up to monoidal natural transformations we apply Lemma \ref{kernel1} for the element $x=\eta^{-1}\cdot t$ in $H$: We only have to check that indeed
$$\phi(h)=tht^{-1}=xhx^{-1}$$
since $\eta\in \Reg^1_L(H^*)$ is by definition (convolution-) central in $(H^*)^*=H$, and that 
$$\sigma = \d(\eta^{-1})=(\d\eta)^{-1}=\Delta(\eta)(\eta^{-1} \otimes \eta^{-1})= \Delta(x)(x^{-1} \otimes x^{-1})$$
since $t$ is grouplike. 
\end{proof}

In particular, the subgroup $\Aut_{Hopf}(H)$ is mapped to monoidal autoequivalences of the form $(F_\phi,J^{triv})$ given by a pullback of the $H$-module action along an $\phi \in \Aut_{Hopf}(\DG)$ and trivial monoidal structure (see Example \ref{fphi}). The subgroup $\Z_L^2(H^*)$ is mapped to monoidal autoequivalences of the form $(\id,J^{\sigma})$, which act trivial on objects and morphisms but have a non-trivial monoidal structure given by $J^\sigma$ (see Corollary \ref{jsigma}).
Note that according to Lemma \ref{kernel1} there are in general invertible $x\in H$ that are not group-likes but still give functors $\Psi(\phi,\sigma)$ that are trivial up to monoidal natural transformations and which are not zero in $\Out_{Hopf}(H) \ltimes \H^2_L(H^*)$.

\begin{corollary} We call the image of map (\ref{outh2}) the group of \emph{lazy monoidal autoequivalences} $\Aut_{mon,L}(H\md\mod)$ and get an short exact sequence 
 \begin{align}\label{monlazy} 0 \to \mathrm{Int}(H)/\mathrm{Inn}(H) \to  \mathrm{Out}_{Hopf}(H) \ltimes \H^2_L(H^*) \to \Aut_{mon,L}(H\md \mod) \to 0 \end{align}
Further, there is an embedding of groups $\Out_{Hopf}(H) \to \Aut_{mon,L}(H \md \mod)$ and we have an bijection of sets of cosets 
\begin{align}\label{monalmoustlazy} \H^2_{aL}(H^*) \simeq \Aut_{mon,L}(H\md\mod)/\Out_{Hopf}(H) \end{align} 
\end{corollary}
\begin{proof}
For the claim (\ref{monalmoustlazy}) we use the exact sequence to define the map from right to left by $[\phi,\sigma] \mapsto [\sigma]$ which is well-defined because a different representative in $[\phi,\sigma]$ would be mapped to a $2$-cocycle that differs by an almost lazy coboundary from $\sigma$. The rest is easy to check.     
\end{proof}

\section{Monoidal autoequivalences of $\Rep(G)$}
\label{seckupg}
We want to determine the lazy cohomology on $k^G$. We obtain this by using Movshev's classification of $k^G$-Galois objects \cite{Mov93} and the additional result in the form presented in Davydov \cite{Dav01} and apply these to the lazy Galois objects. Further, we give an explicit $2$-cocycle representing a lazy cohomology class on $k^G$. \\

Let us recall that for a finite group $G$ a left $G$-algebra (equivalently a right $k^G$-comodule algebra) is an associative algebra $R$ together with a left $G$-action given by a homomorphism $G \rightarrow \Aut(R)$. A $G$-algebra is Galois if the algebra homomorphism 
\begin{align*}
\theta: \;R\otimes kG \rightarrow \End(R) \quad\quad \theta(r\otimes g)(r') = r(g.r')
\end{align*}
is an isomorphism. Therefore, a left $G$-Galois algebra is a right $k^G$-Galois object. For any group 
$S$ and any $2$-cocycle $\eta \in Z^2(S)$ we view the twisted group algebra $k_\eta S$ as a $S$-algebra with the associative multiplication given by $g \cdot h = \eta(g,h)gh$ and the left $S$-action $g.h = g \cdot h \cdot g^{-1}$. 
Given a subgroup $S$ of $G$ and a $S$-algebra $M$ there is a natural way to construct a $G$-algebra by induction: 
\[ ind^S_G(M) := \{ r:G \rightarrow B \mid r(sg) = s.r(g)  \}  \]
which is an associative algebra with the pointwise multiplication of maps and has a left $G$-action given by $(g'.r)(g) = r(gg')$. We are now ready to present the first classification result.  

\medskip

\begin{lemma}\label{lm_DavydovClassification}(\cite{Dav01} Theorem 3.8) \\
Let $k$ be an algebraically closed field of characteristic zero. Then there is a bijection between isomorphism classes of right $k^G$-Galois objects and conjugacy classes $(S,[\eta])$ where $S$ is a subgroup of $G$ and $[\eta] \in \H^2(S,k^\times)$ for $\eta$ a non-degenerate $2$-cocycle. The isomorphism assigns to a conjugacy class $(S,[\eta])$ the isomorphism class of \[ R(S,\eta) :=  \{ r:G \rightarrow k_{\eta}S \mid r(sg) = s.r(g)\; \forall s \in S,g \in G \} \] with multiplication given by pointwise multiplication of functions and with a $G$-action given by $(g.r)(h)=r(hg)$ or equivalently a $k^G$-coaction given by $ \delta (r) = \sum_g e_g \otimes g.r$.    
\end{lemma}

\medskip

\begin{lemma}\label{lm_GPrimeIff}(\cite{Dav01}: Proposition 6.1) \\
Let $R(S,\eta)$ be a right $k^G$-Galois object as above and let $G':=\Aut_G(R)$. The following are equivalent: 
\begin{itemize} 
\item $R(S,\eta)$ is an $k^{G'}$-$k^{G}$-Bigalois object with the natural left $k^{G'}$-coaction resp. right $G'$-action $r.f = f(r)$.  
\item $|G'|=|G|$.    
\item $S$ is an abelian normal subgroup of $G$ and the class $[\eta]$ is $G$-invariant.     
\end{itemize}
\end{lemma}

Relevant to this article is to find all Bigalois objects with $G\cong G'$ from this description. We now determine for the previously introduced Bigalois objects the explicit $2$-cocycles:

\begin{lemma}\label{lm_BuenosAires}
Let $S$ be a normal abelian subgroup of $G$ and $\eta \in \Z^2(S)$ a non-degenerate $2$-cocycle i.e. $\langle s,t\rangle:=\eta(s,t)\eta(t,s)^{-1}$ is a non-degenerate bimultiplicative symplectic form. Further, assume for simplicity that $\eta(s,s^{-1}) = 1$ for all $s \in S$ (Note that this is always possible up to cohomology). Then there is an isomorphism of $k^G$-comodule algebras: $(k^G)_{\alpha}\simeq R(S,\eta)$ where $\alpha \in \Z^2(k^G)$ is defined by $$\alpha(f,f')=\frac{1}{|S|^2}\sum_{r,t,r',t'\in S}\eta(t,t')\langle t,r\rangle\langle t',r'\rangle f(r)f'(r') $$ 
\end{lemma}
\begin{proof}
We show that an isomorphism $\Phi:(k^G)_{\alpha} \mapsto R(S,\eta)$ is given by
$$\Phi:f\mapsto \frac{1}{|S|}\left(h\mapsto \sum_{t,r\in S}t f(rh) \langle t,r\rangle \right)$$     
We check that $\Phi(f)$ is a $kS$-linear map:
\begin{align*}
	\Phi(f)(sg)
	&=\frac{1}{|S|}\sum_{t,r\in S}tf(rsg)\langle t,r\rangle 
	 =\frac{1}{|S|}\sum_{t,r'}tf(r'g)\langle t,r'\rangle\langle t,s\rangle^{-1}\\
	&=\frac{1}{|S|}\sum_{t,r'}\left(\eta(s,t)t\eta(t,s)^{-1}\right)f(r'g)\langle t,r'\rangle\\
	&=\frac{1}{|S|}\sum_{t,r'}\left(\eta(s,t)t\frac{\eta(st,s^{-1})}{\eta(t,1)\eta(s,s^{-1})}\right)f(r'g)
	\langle t,r'\rangle\\
	&=\frac{1}{|S|}\sum_{t,r'}\left(s.t\right)f(r'g)\langle t,r'\rangle
	=s.\Phi(f)(g)
\end{align*}
Next we check this is a $kG$-module morphism:
\begin{align*}
	\Phi(h.f)
	 =\Phi(f_{2}(h)f_{1})
	=\Phi(g'\mapsto f(g'h)) 
	 =\frac{1}{|S|}(g\mapsto \sum_{t,r\in S}tf(rgh)\langle t,r\rangle )
	=h.\Phi(f)
\end{align*}
Further, it is an algebra morphism: $\Phi(1)(g)
	=\frac{1}{|S|^2}\sum_{t,r\in S}t \langle t,r\rangle
	=\sum_{t\in S}t \delta_{t,1}=1$
\begin{align*}
	\left(\Phi(f)\Phi(f')\right)(g)
	&=\frac{1}{|S|^2}\left(\sum_{t,r\in S}tf(rg)\langle t,r\rangle\right)\cdot_\eta
	\left(\sum_{t',r'\in S}t'f(r'g)\langle t',r'\rangle\right)\\
	&=\frac{1}{|S|^2}\sum_{t,r,t',r'\in S}tt'\eta(t,t')\cdot 
	f(rg)\langle t,r\rangle f'(r'g)\langle t',r'\rangle\\
	&=\frac{1}{|S|^2}\sum_{r,t,r',t'\in S} 
	\left(\frac{1}{|S|}\sum_{\tilde{t},\tilde{r}\in S}\tilde{t}\langle\tilde{t},\tilde{r}\rangle
	\langle t,\tilde{r}^{-1}\rangle\langle t',\tilde{r}^{-1}\rangle\right)\\
	&\cdot f_{1}(r)f'_{1}(r')\langle t,r\rangle\langle t',r'\rangle
	\eta(t,t')
	\cdot f_{2}(g)f'_{2}(g)\\
	&=\frac{1}{|S|^3}\sum_{\tilde{t},\tilde{r},r,t,r',t'\in S} \tilde{t}\langle\tilde{t},\tilde{r}\rangle\\
	&\cdot f_{1}(r)f'_{1}(r')\langle t,r\tilde{r}^{-1}\rangle\langle t',r'\tilde{r}^{-1}\rangle
	\eta(t,t')
	\cdot f_{2}(g)f'_{2}(g)\\
	\substack{r''=r\tilde{r}^{-1}\\r'''=r'\tilde{r}^{-1}}\quad
	&=\frac{1}{|S|^3}\sum_{\tilde{t},\tilde{r},r'',t,r''',t'\in S} \tilde{t}\langle\tilde{t},\tilde{r}\rangle\\
	&\cdot f_{1}(r'')f'_{1}(r''')\langle t,r''\rangle\langle t',r'''\rangle\eta(t,t')
	\cdot f_{2}(\tilde{r})f'_{2}(\tilde{r})
	\cdot f_{3}(g)f'_{3}(g)\\
	&=\frac{1}{|S|}\sum_{\tilde{t},\tilde{r}\in S} \tilde{t}\langle\tilde{t},\tilde{r}\rangle
	\cdot \alpha(f_{1},f'_{1})f_{2}(\tilde{r}g)f'_{2}(\tilde{r}g) 
	=\Phi(f\cdot_\alpha f')(g)
\end{align*} 
We finally check bijectivity of $\Phi$. We first note that for any coset $Sg\subset G$ the functions $f$ which are nonzero only on $Sg$ are sent to functions $\Phi(f)$ which are nonzero only on $Sg$. With the fixed representative $g$ of a coset $Sg$ we consider the basis $e_{sg}$ for $k^{Sg}$. By construction this element is mapped to the following element in $\Hom_{kS}(k[Sg],kS)$:
\begin{align*}
\Phi(e_{sg})
&=\frac{1}{|S|}\left(s'g\mapsto \sum_{t,r\in S}t e_{sg}(rs'g)\langle t,r\rangle\right)
=\frac{1}{|S|}\left(s'g\mapsto \sum_{t\in S}t\langle t,s\rangle\langle t,s'\rangle^{-1}\right)\\
\end{align*}
This is by construction a Fourier transform with a non-degenerate form $\langle,\rangle$, which implies bijectivity. More explicitly, we show injectivity by considering $s'=1$, then we get elements $\sum_{t\in S}t\langle t,s\rangle$ in $S$, which are linearly independent by the assumed non-degeneracy. Now bijectivity follows because source and target have dimension $|S|$. 
\end{proof}

In particular, the assumption in Lemma \ref{lm_GPrimeIff} is that the class $[\eta]\in \H^2(S)$ is $G$-invariant and hence the alternating bicharacter (symplectic form) $\langle,\rangle$ on $S$ is $G$-invariant. Then the criterion in Example \ref{exm_LazyIff} 2 shows easily $\alpha$ is lazy iff it is $G$-invariant and an easy calculation shows:

\begin{corollary}
A lazy cocycle for $k^G$ is up to cohomology the restriction of a $G$-invariant $2$-cocycle $\eta$ (not just an $G$-invariant class $[\eta]$) on a normal abelian subgroup $S$ in the sense of Lemma \ref{lm_BuenosAires}. \\
\end{corollary}

\begin{corollary}
Let $\alpha\in \Z^2_L(k^G)$ be a $2$-cocycle with the additional property $\alpha(e_g,e_h)=\alpha(e_h,e_g)$. Then $\alpha$ is cohomologically trivial.  
\end{corollary}	
\begin{proof}
Since $k^G$ is commutative, we have $\d\nu(e_x,e_y)=\nu(e_{x_1})\nu(e_{y_1})\nu^{-1}(e_{x_2}e_{y_y}) = \d\nu(e_y,e_x)$, thus every $2$-cocycle $\alpha'$ in the same cohomology class as $\alpha$ is also symmetric, hence $\alpha'(e_g,e_h)=\alpha'(e_h,e_g)$. By Lemma \ref{lm_BuenosAires} every cohomology class has a representative induced by an $\eta \in \Z^2(S)$ for a normal abelian subgroup $S$. One checks that such an $\eta$ is given by 
$$\eta(s,s') = \sum_{x,y \in S} \alpha(e_x,e_y) \langle x,s \rangle \langle y,s' \rangle $$ 
which implies that $\eta$ is symmetric on $S$. Since $S$ is abelian $\eta$ is cohomologically trivial. 
Since any lazy $2$-cocycle is up to cohomology the restriction of a $2$-cocycle $\eta$ on a normal abelian subgroup $S\subset G$ the symmetry condition implies that $\eta$ is also symmetric, hence $\langle x,s \rangle = 1$ for all $x,s \in S$. But we know from Lemma \ref{lm_DavydovClassification} that $\eta$ is non-degenerate hence $\langle x,s \rangle = 1$ for all $x$ implies $s=1$. This implies that the abelian subgroup $S$ is trivial and hence $\alpha$ is cohomologically trivial.   
\end{proof}

Note that this can also be seen from the fact that the Galois object $_\alpha(k^G)$ is a commutative algebra, since $k^G$ commutative and $\alpha$ symmetric. Such a Galois object is then trivial up to isomorphism.  

\section{A Bruhat-like decomposition of $\Aut_{Hopf}(\DG)$}\label{sec_cell}

\noindent

In this section we give a description of $\Aut_{Hopf}(\DG)$. More precisely, we determine in Theorem \ref{thm_cell} a decomposition of $\Aut_{Hopf}(\DG)$ into double cosets similar to the Bruhat-decomposition of a Lie group. It $G$ is not elementary abelian the reflections that we need for our \emph{double} coset decomposition are twisted, in particular they do not square to the identity. However it is possible to give a coset decomposition based on non-twisted reflections.

Our results in this section rely on the approach \cite{ABM} Corollary 3.3 and on the works of Keilberg \cite{Keil}. He has determined a product decomposition (exact factorization) of $\Aut_{Hopf}(\DG)$ whenever $G$ does not contain abelian direct factors. In \cite{KS14} Keilberg and Schauenburg determined $\Aut_{Hopf}(\DG)$ in the general case, hence when $G$ is allowed to have abelian direct factors using an approach that differs from ours.   

\medskip  

\begin{proposition}(\cite{Keil} 1.1, 1.2)\\
The underlying set of $\Aut_{Hopf}(\DG)$ is in bijection with the set of invertible matrices $\begin{pmatrix} u & b \\ a & v \end{pmatrix}$ where 
\begin{equation*} 
\begin{split} 
&u:k^G \rightarrow k^G    \hspace{1cm} \text{is a Hopf algebra morphism    } \\
&b:G \rightarrow \hat{G}  \hspace{1cm} \text{is a group homomorphism   } \\
&a:k^G \rightarrow kG     \hspace{1cm} \text{is a Hopf algebra morphism    } \\ 
&v:G \rightarrow G        \hspace{1cm} \text{is a group homomorphism   } 
\end{split}
\end{equation*}  
fulfilling the following three additional conditions for all $f \in k^G$ and $g\in G$: 
\begin{align}
u(f_1) \times a(f_2)=u(f_2) \times a(f_1) \quad\;
v(g) \triangleright u(f) = u(g \triangleright f) \quad\;
v(g)a(f) = a(g \triangleright f)v(g)
\label{eqn:threepr}
\end{align}
This bijection maps such a matrix to an automorphism $\phi \in \Aut_{Hopf}(\DG)$ defined by: 
\begin{equation*} 
\phi(f \times g) = u(f_1)b(g) \times a(f_2)v(g) \hspace{1cm} \forall f \in k^G \hspace{0.3cm} \forall g \in G 
\end{equation*}   
We equip this set of matrices with matrix multiplication where we use convolution to add and composition to multiply then the above bijection is a group homomorphism. 
\label{propone}
\end{proposition}
\begin{example} For $G=A$ a finite abelian group we get an isomorphism $$\Aut_{Hopf}(DA) \simeq \Aut(\hat{A} \times A)$$ where all maps $u,b,a,v$ are group homomorphisms. \\
\end{example}

\noindent
In the following we will often express the maps $u,a,b$ in terms of a canonical basis in the following way    
\begin{equation*}
\begin{split}
&u(e_g) = \sum_{h \in G} u(e_g)(h) e_h \hspace{1cm} b(g) = \sum_{h \in G} b(g)(h) e_h \\
&a(e_g) = \sum_{h \in G} e_h (a (e_g)) h =: \sum_{h \in G} a^h_g h
\end{split}
\end{equation*}
We denote by e.g. $u^*:kG \rightarrow kG$ the dual map of $u:k^G \to k^G$, hence $e_h(u^*(g)) = u(e_h)(g)$ and similarly for $a,b,v$.  
An automorphism $\phi \in \Aut_{Hopf}(\DG)$ can then be given in the following way: 
\[ \phi(e_g\times h) = \sum_{x,y \in G} b(h)(x) \hspace{0.05 cm} a_{gu^*(x)^{-1}}^{v(h)^{-1}y} \hspace{0.05 cm} (e_x \times y)  \]

In the matrix notation $\left(\begin{smallmatrix} u & b \\ a & v \end{smallmatrix}\right)$ we use the following conventions: 
\begin{itemize}
\item $u \equiv 0$ denotes the map $k^G \rightarrow k^G; e_g \mapsto (h \mapsto \delta_{g,1_G})$ and $u \equiv 1$ the identity map on $k^G$.
\item $b \equiv 0$ denotes the map $kG \rightarrow k^G;g \mapsto 1_{k^G}$ 
\item $a \equiv 0$ denotes the map $k^G \rightarrow kG;e_g \mapsto 1_G \delta_{g,1_G}$
\item $v \equiv 0$ denotes the map $G \rightarrow G; g \mapsto 1_G$ and $v \equiv 1$ the identity map on $G$. \\  
\end{itemize} 
 
\medskip 

\begin{lemma}(\cite{Keil} Sect. 2)~\\ Let $\left(\begin{smallmatrix} u & b \\ a & v \end{smallmatrix}\right)\in \Aut_{Hopf}(\DG)$. Then the following holds: 
\begin{itemize}
\item The Hopf morphism $a$ is uniquely determined by a group isomorphism $\hat{A} \cong B$ where $A,B$ are abelian subgroup of $G$ such that $\im(a^*)=kA$ and $\im(a)=kB$. The map $a$ is then given by composing $ k^G \hookrightarrow k^A \cong k\hat{A} \cong kB \hookrightarrow kG $.
\item $A,B \leq Z(G)$, $G=Z(G)\im(v)$, $\im(a)\im(v) = \im(v)\im(a)$.   
\item $u^* \circ v$ is a normal group homomorphism. 
\item The kernels of $v$ and $u^*$ are contained in an abelian direct factor of $G$. 
\end{itemize}
\label{lm_KeilbergTechnical}
\end{lemma}

\noindent
We now introduce several important subgroups of $\Aut_{Hopf}(DG)$. The first subgroup illustrates how a group automorphism of $G$ induces an Hopf automorphism of $\DG$. \\

\begin{proposition}(\cite{Keil} 4.3) \\ 
There is a natural subgroup of $\Aut_{Hopf}(\DG)$ given by:  
\begin{equation*} 
V := \left\{ \begin{pmatrix} (v^{-1})^* & 0 \\ 0 & v \end{pmatrix} \mid v \in \Aut(G) \right\}
\end{equation*}
An element in $V$ corresponds to the following automorphism of $\DG$: 
\[ e_g \times h \mapsto e_{v(g)}\times v(h) \] 
We obviously have an isomorphism of groups: $V \simeq \Aut(G)$. 
\label{auto}
\end{proposition}

\medskip

Then we have two group of 'strict upper triangular matrices' and 'strict lower triangular matrices' which come from the abelianization $G_{ab}=G/[G,G]$ and the center $Z(G)$ respectively.  

\medskip

\begin{proposition}(\cite{Cour12} 2.3.5) \\
There is a natural abelian subgroup of $\Aut_{Hopf}(\DG)$ given by: 
\begin{equation*} 
B := \left\{ \begin{pmatrix} 1 & b \\ 0 & 1 \end{pmatrix} \mid b \in \Hom(G_{ab},\widehat{G}_{ab}) \right\}
\end{equation*} 
An element in $B$ corresponds to the following isomorphism of $\DG$: 
\[e_g\times h \mapsto b(h)(g)\;e_g \times h \] 
We have an isomorphism of groups $B \cong \Hom(G_{ab},\widehat{G}_{ab})\cong \widehat{G}_{ab}\otimes \widehat{G}_{ab}$ where the $\Hom$-space is equipped with the convolution product and the tensor product is over $\ZZ$.
\end{proposition}

\medskip

\begin{proposition}(\cite{Keil} 4.1, 4.2) \\
There is natural abelian subgroup of $\Aut_{Hopf}(\DG)$ given by: 
\begin{equation*} 
E := \left\{ \begin{pmatrix} 1 & 0 \\ a & 1 \end{pmatrix} \mid a \in \Hom(\widehat{Z(G)},Z(G)) \right\}
\end{equation*}
An element in $E$ corresponds to the following isomorphism of $\DG$: 
\[ e_g\times h \mapsto \sum_{g_1g_2 =g}e_{g_1} \times a(e_{g_2})h \] 
We have an isomorphism $E \cong \Hom(\widehat{Z(G)},Z(G))\cong Z(G)\otimes Z(G)$ where the $\Hom$-space is equipped with the convolution product and the tensor product is over $\ZZ$.
\end{proposition}

\medskip

\begin{proposition}(\cite{Keil} 4.5) \\ 
There is a subgroup of $\Aut_{Hopf}(\DG)$ given by: 
\begin{equation*}
V_c:= \left\{ \begin{pmatrix} (v^{-1})^* & 0 \\ 0 & 1 \end{pmatrix} \mid v\in \Aut_c(G) \right\}
\end{equation*}
where $\Aut_c(G) := \{ v \in \Aut(G) | v(g)g^{-1} \in Z(G) \}$. An element in $V_c$ corresponds to the following map 
\begin{equation*}
e_g\times h \mapsto e_{v(g)}\times h 
\end{equation*}
\label{autc}
\end{proposition}

\noindent
We can now state the main result of \cite{Keil} which determines $\Aut_{Hopf}(\DG)$ in the case that $G$ is purely non-abelian (i.e. has no direct abelian factors): 

\medskip

\begin{theorem}(\cite{Keil} Theorem 5.7)\label{thm_keilberg}\\
Let $G$ be a purely non-abelian finite group. There is an exact factorization into subgroups
\begin{equation*}
\begin{split}
\Aut_{Hopf}(\DG) &\simeq E((V_c \rtimes V) \ltimes B) \\
        &\simeq ((V_c \rtimes V) \ltimes B)E
\end{split}
\end{equation*}
\end{theorem}

\noindent
The main step in the proof is using the fact that $\ker(u^*)$ and $\ker(v)$ are contained in direct factors of $G$ according to Lemma \ref{lm_KeilbergTechnical}. Clearly, if $G$ is purely non-abelian then $u,v$ have to be invertible. This leads directly to the above decomposition. \\  
The exact factorization fails to be true in the presence of direct abelian factors. In this case neither $u$ nor $v$ have to be invertible, but their kernels are still contained in an direct abelian factor. From this point of view it seems natural to introduce an additional class of automorphisms of $\DG$ that act on direct abelian factors of $G$. These will be maps that exchange an abelian factor $kC \subset k^G\rtimes kG$ with its dual $k^C \subset k^G \rtimes kG$.  

\medskip 

\begin{proposition} 
Let $R_t$ be the set of all tuples $(H,C,\delta,\nu)$, where $C$ is an abelian subgroup of $G$ and $H$ is a subgroup of $G$, such that $G = H \times C$, $\delta:kC \stackrel{\sim}{\rightarrow} k^C$ a Hopf isomorphism and $\nu:C \rightarrow C$ a nilpotent homomorphism.    
\begin{itemize}
\item[(i)] For $(H,C,\delta,\nu)$ we define a \emph{twisted reflection} $r_{(H,C,\delta,\nu)}:\DG \rightarrow \DG$ of $C$ by: 
\begin{equation*}
\begin{split}
(f_H,f_C) \times (h,c) &\mapsto (f_H,\delta(c)) \times (h, \delta^{-1}(f_C)\nu(c))
\end{split}
\end{equation*} 
We write $p_H,p_C$ for the two projections to $H,C$ and $\iota_H,\iota_C$ for the two embeddings. Then the matrix corresponding to a twisted reflection is given by 
\begin{equation*}
\begin{pmatrix} (\iota_H \circ p_H)^* & p_C^* \circ \delta \circ p_C \\ \iota_C \circ \delta^{-1} \circ \iota_C^* & p_H+\nu \end{pmatrix} 
\end{equation*}
All twisted reflections are Hopf automorphisms. 
\item[(ii)] Denote by $R$ the subset of $R_t$ elements with $\nu=1_C$. We call the corresponding Hopf automorphisms \emph{reflections} of $C$. For two triples $(H,C,\delta)$ and $(H',C',\delta')$ in $R$ with $C\cong C'$ the elements $r_{(H,C,\delta)},r_{(H',C',\delta')}$ are conjugate in $\Aut_{Hopf}(\DG)$ (by an element in $V\cong \Aut(G)$).   
\end{itemize}      
\label{reflections}
\end{proposition}
\noindent
In the following we will sometimes restrict ourselves to a set of representatives $r_{[C]}$ for each isomorphism type $[C]$ of $C$. In order to simplify the notation we abbreviate the matrix corresponding to a (twisted) reflection $(H,C,\delta,\nu)$ by 
$\left( \begin{smallmatrix} \hat{p}_H & \delta \\ \delta^{-1} & p_H+\nu \end{smallmatrix}\right)$. 

\begin{proof} For convenience we denote $g_H := p_H(g)$, $g_C := p_C(g)$ etc. \\
(i) In order to prove that a reflection is indeed an automorphism of $\DG$ we have to check bijectivity (which is clear) and the three equations (\ref{eqn:threepr}). The first equation holds because for all $x,g \in G$ we have: 
\begin{equation*}
\begin{split}
(\delta^{-1} \circ \iota_C^*)(e_{p_H(x)^{-1}g}) = \delta^{-1}(e_{x_H^{-1}g} \circ \iota_C) = \delta_{x_H,1}\delta_{g_H,1}\delta^{-1}(e_{g_C}) = (\delta^{-1} \circ \iota_C^*)(e_{gp_H(x)^{-1}})
\end{split}
\end{equation*}
For the second equation of $(\ref{eqn:threepr})$ we have to check that with $u = \hat{p}_H$ and $v = p_H + \nu$  we get $u(e_{y})(v(g)^{-1}xv(g)) = u(e_{gyg^{-1}})(x)$ for all $x,y,g \in G$. This is indeed true:  
\begin{align*}
&e_y( u^*(v(g)xv(g)))  = e_y(p_H( g^{-1}xg)) = \delta_{g_Cy_Cg_C^{-1},1}e_{g_Hy_Hg_H^{-1}}(x_H) = e_{gyg^{-1}}(p_H(x))																					
\end{align*}   
The last equation also holds since $g_H \nu(g_C) \delta^{-1}(e_x \circ \iota_C) = \delta^{-1}(e_{gxg^{-1}} \circ \iota_C) g_H \nu(g_C) $. \\

\noindent
(ii) Let $H,C,H',C'$ be subgroups of $G$ such that $H\times C=G=H'\times C'$ and $v_C:C\cong C'$ (such that $v_C^*\delta=\delta'v_C$). Then we claim that there is an automorphism $v\in \Aut(G)$ such that $v(H)=H'$ and $v(C)=C'$ such that $v|_C=v_C$ (this is a known fact). This immediately implies that $v$ conjugates $r_{H,C,\delta}$ to $r_{H',C',\delta'}$ as claimed.\\

We now give prove the claim that given an $\varphi:H\times C\cong H'\times C'$ and an isomorphism $v_C:C\to C'$ then we also have an isomorphism $H\cong H'$ (note that for infinite groups this statement is \emph{wrong}, as for example we can have $C=C'=H\times H'\times H\times H' \times \cdots$).\\ Let us write $\varphi$ as a matrix:
$$\varphi=\begin{pmatrix}\varphi_{H,H'} & \varphi_{C,H'} \\ \varphi_{H,C'}&\varphi_{C,C'}\end{pmatrix}$$
\begin{itemize}
\item If $\varphi_{C,C'}$ is invertible we can find an isomorphism $\varphi':H\times C\cong H'\times C'$ that is diagonal, hence with $\varphi'_{H,C'}=0=\varphi'_{C,H'}$ by
$$\varphi':=
\begin{pmatrix}1 & -\varphi_{C,H'}\varphi_{C,C'}^{-1} \\ 0 & 1\end{pmatrix}
\begin{pmatrix}\varphi_{H,H'} & \varphi_{C,H'} \\ \varphi_{H,C'}&\varphi_{C,C'}\end{pmatrix}
\begin{pmatrix}1 & 0 \\ -\varphi_{C,C'}^{-1}\varphi_{H,C'}&1\end{pmatrix}$$
which implies that $\varphi'_{H,H'}:H\cong H'$ is an isomorphism. 
\item  If $\varphi_{C,C'}$ is not invertible we assume w.l.o.g (induction otherwise) that $C=\ZZ_{p^n}$. In this case, $\varphi_{C,H'}$ has to be injective (and $\varphi_{H,C}$ has to be surjective) in order for $\varphi$ to be bijective. Again by column-transformation we thus obtain an isomorphism $\varphi'$ such that $\varphi'_{C,C'}=0$. We thus obtain a direct factor $\varphi'_{C,H'}(C)$ in $H'$ and a direct factor $\varphi_{H,C'}^{'-1}(C)$ of $H$. We will now write $\varphi'$ as a $3\times 3$-matrix  
$$\varphi'=\begin{pmatrix} \phi'_{\tilde{H}',\tilde{H}'} & * & 0 \\ * & * & \varphi_{C,H'} \\ 0 & \varphi_{H,C'} & 0 \end{pmatrix}
:\;\tilde{H}\times \varphi_{H,C'}^{'-1}(C)\times C\to \tilde{H}'\times \varphi'_{C,H'}(C) \times C$$
Again, we can eliminate the $*$ hence also get an isomorphism $H\cong H'$ as above.
\end{itemize}
 \end{proof}

\medskip

\begin{theorem}~\label{thm_cell}
\begin{itemize}
\item[(i)] Let $G$ be a finite group, then $\Aut_{Hopf}(\DG)$ is generated by the subgroups $V$, $V_c$, $B$, $E$ and the set of reflections $R$. 
\item[(ii)] For every $\phi \in \Aut_{Hopf}(\DG)$ there is a twisted reflection $r = r_{(H,C,\delta,\nu)} \in R_t$ such that $\phi$ is an element in the double coset  
\begin{equation*} \begin{split}
&[(V_c \rtimes V) \ltimes B] \cdot r \cdot [(V_c \rtimes V) \ltimes E]
\end{split}
\end{equation*}
\item[(iii)] Two double cosets corresponding to reflections $(C,H,\delta),(C',H',\delta') \in R$ are equal if and only if $C \simeq C'$.   
\item[(iv)] For every $\phi \in \Aut_{Hopf}(\DG)$ there is a reflection $r=r_{(C,H,\delta)} \in R$ such that $\phi$ is an element in 
\begin{equation*} \begin{split}
&r \cdot [ B ((V_c \rtimes V) \ltimes E)]
\end{split}
\end{equation*}
\item[(v)] For every $\phi \in \Aut_{Hopf}(\DG)$ there is a reflection $r = r_{(C,H,\delta)} \in R$ such that $\phi$ is an element in 
\begin{equation*} \begin{split}
& [((V_c \rtimes V) \ltimes B)E] \cdot r 
\end{split}
\end{equation*}
\end{itemize}
\end{theorem}

\medskip

\noindent
Before we turn to the proof we illustrate the statement of Theorem \ref{thm_cell} on some examples:

\medskip

\begin{example}
In the case $G$ is purely non-abelian, there are no (non-trivial) reflections. We get the result of Theorem \ref{thm_keilberg}. 
\end{example}

\begin{example}\label{exm_Fp_Aut}
	For $G=(\F_p^n,+)$ a finite vector space, we have directly
	$$\Aut_{Hopf}(\DG)=\Aut(\F_p^n\times \F_p^n)\cong \GL_{2n}(\F_p)$$
	On the other hand the previously defined subgroups are in this case:
	\begin{itemize}
	\item $V\cong \Aut(G) = \GL_n(\F_p)$ and $V_c\ltimes V\cong \GL_n(\F_p)\times \GL_n(\F_p)$
	\item $B\cong \widehat{G}_{ab}\otimes_\ZZ \widehat{G}_{ab}=\F_p^{n \times  n}$ as 
	additive group.
	\item $E\cong Z(G)\otimes_\ZZ Z(G)=\F_p^{n \times n}$ as additive group.
	\end{itemize}
	The set $R$ is very large: For each dimension $d\in\{0,\ldots, n\}$ there is a unique isomorphism type $C\cong \F_p^d$. The possible subgroups of this type $C\subset G$ are 
	the Grassmannian $\mbox{Gr}(n,d,G)$, the possible $\delta:C\to \hat{C}$ are parametrized by $\GL_d(\F_p)$ and in this fashion $R$ can be enumerated. \\
	On the other hand, we have only $n+1$ representatives $r_{[C]}$ for each dimension $d$, given for example by permutation matrices 
	\begin{equation*} 
\left(
\begin{array}{cc|cc}    
                    0 & 0 &  \mathbbm{1}_{d} & 0 \\ 
                    0 & \mathbbm{1}_{n-d} &  0 & 0 \\ \hline 
										\mathbbm{1}_{d} & 0 & 0 & 0  \\ 
										0 & 0 &  0 & \mathbbm{1}_{n-d} \\
\end{array}
\right)
\end{equation*}
One checks this indeed gives a decomposition of $\GL_{2n}(\F_p)$ into $V_cVB$-$V_cVE$-cosets, e.g.
\begin{center}
\begin{tabular}{rrccccc}
$\GL_{4}(\F_p)$ & $=$ & $(V_cVB\cdot r_{[1]}\cdot V_cVE)$ & $\cup$ & $(V_cVB\cdot r_{[\F_p]}\cdot V_cVE)$ & $\cup$ & $(V_cVB\cdot r_{[\F_p^2]}\cdot V_cVE)$ \\
 $|\GL_{4}(\F_p)|$ & $=$ & $p^8|\GL_2(\F_p)|^2$ 
 &$+$ &  $\frac{p^3|\GL_2(\F_p)|^4}{(p-1)^4}$ 
 &$+$ & $p^4|\GL_2(\F_p)|^2$ \\
  & $=$ & $p^8(p^2-1)^2(p^2-p)^2$ 
 & $+$ &  $\frac{p^3(p^2-1)^4(p^2-p)^4}{(p-1)^4}$ 
 & $+$ & $p^4(p^2-1)^2(p^2-p)^2$ \\
  & $=$ & \multicolumn{3}{l}{$(p^4-1)(p^4-p)(p^4-p^2)(p^4-p^3)$} 
 &  & 
\end{tabular}
\end{center}
It corresponds to a decomposition of the Lie algebra $A_{2n-1}$ according to the 
$A_{n-1}\times A_{n-1}$ parabolic subsystem. Especially on the level of Weyl 
groups we have a decomposition as double cosets of the parabolic Weyl group
\begin{align*}
    \SS_{2n}
    &=(\SS_n\times \SS_n)1(\SS_n\times \SS_n)
    \;\cup\; (\SS_n\times \SS_n)(1,1+n)(\SS_n\times \SS_n)\;\cup\\
    & \;\cdots
    \;\cup \;(\SS_n\times \SS_n)(1,1+n)(2,2+n)\cdots (n,2n)(\SS_n\times 
\SS_n)\\
    e.g.\quad |\SS_4|
    &=4+16+4
\end{align*}

In this case, the full Weyl group $\SS_{2n}$ of $\GL_{2n}(\F_p)$ is the set of all 
reflections (as defined above) that preserve a given decomposition 
$G=\F_p\times \cdots\times \F_p$.
\end{example}

\medskip

\begin{proof}[Proof of Theorem \ref{thm_cell}]~\\
\noindent
(i) follows immediately from (iv).\\
\noindent
(ii) From above we know that $\ker(v)$ is contained in an abelian direct factor $G$. The other factor can be abelian or not, but we can decompose it into a purely non-abelian factor times an abelian factor. Hence we arrive at the decomposition $G = H \times C$ where $H$ is purely non-abelian and where $\ker(v)$ is contained in an direct abelian factor of $C$.   
Since $C$ is a finite abelian group there is an $n \in \nat$ and an isomorphism \[ C \cong C_1 \times ... \times C_n \] where $C_i$ are cyclic groups of order $p_i^{k_i}$ for some prime numbers $p_i$ and $k_i \in \nat$ with $p_i^{k_i} \leq p_j^{k_j}$ for $i\leq j$.  \\
A general Hopf automorphism $\phi \in \Aut_{Hopf}(\DG)$ can then be written in matrix form with respect to the decomposition $G = H \times C_1 \times C_2 \times \cdots \times C_n$ as   
\begin{equation}
\phi=
\left(
\begin{array}{ccc|ccc}
 u_{H,H}   &\cdots      & u_{C_n,H}         & b_{H,H}   &\cdots  & b_{C_n,H}         \\
\vdots   &\ddots      &\vdots             &\vdots   &\ddots   &\vdots           \\ 
 u_{H,C_n} &\cdots      & u_{C_n,C_n}       & b_{H,C_n} &\cdots  & b_{C_n,C_n}       \\  \hline
a_{H,H}   &\cdots      & a_{C_n,H}         & v_{H,H}   &\cdots  & v_{C_n,H}              \\ 
 \vdots   &\ddots      &\vdots             &\vdots   &\ddots   &\vdots           \\ 
a_{H,C_n} &\cdots      & a_{C_n,C_n}       & v_{H,C_n} &\cdots  & v_{C_n,C_n}                  
\end{array}
\right)
\label{eqn:cn}
\end{equation}
Let $c_n$ be the generator of $C_n$. Since $\phi$ is an automorphism we know that the order of $\phi(c_n)$ is also $p_n^{k_n}$.
This implies that one of the $b_{C_n,H}(c_n), \cdots ,b_{C_n,C_n}(c_n),v_{C_n,H}(c_n), \cdots ,v_{C_n,C_n}(c_n)$ has order $p_n^{k_n}$. Then we can have three possible cases: 
\begin{itemize}
\item[] Case (1): One of the $v_{C_n,C_n}$ or $b_{C_n,C_n}$ is injective. 
\item[] Case (2): One of the $v_{C_n,C_m}$ or $b_{C_n,C_m}$ is injective  and $m < n$. 
\item[] Case (3): One of the $v_{C_n,H}$ or $b_{C_n,H}$ is injective. 
\end{itemize}
\noindent
Case $(1)$:  If $v_{C_n,C_n}$ is injective, then it has to be bijective (since $C_n$ is finite). We can construct an element in $B$ and an element in $V_c:$  

\begin{equation}
\left(
\begin{array}{ccc|ccc}
 1       &\cdots      &0          &0      		&\cdots  & 0         			\\
\vdots   &\ddots      &\vdots     &\vdots 		&\ddots  &\vdots          \\ 
0        &\cdots      &1          &0      		&\cdots  & 0        				\\  \hline
0        &\cdots      &0          &1      		&\cdots  & v_{C_n,H}v_{C_n,C_n}^{-1}               \\ 
\vdots   &\ddots      &\vdots     &\vdots 		&\ddots  &\vdots          \\ 
0        &\cdots      &0          &0 					&\cdots  & v_{C_n,C_n}^{-1}                   
\end{array}
\right)
\left(
\begin{array}{ccc|ccc}
 1       &\cdots      &0          &0      		&\cdots  & b_{C_n,H}v_{C_n,C_n}^{-1}         			\\
\vdots   &\ddots      &\vdots     &\vdots 		&\ddots  &\vdots          \\ 
0        &\cdots      &1          &0      		&\cdots  & b_{C_n,C_n}v_{C_n,C_n}^{-1}        				\\  \hline
0        &\cdots      &0          &1      		&\cdots  & 0               \\ 
\vdots   &\ddots      &\vdots     &\vdots 		&\ddots  &\vdots          \\ 
0        &\cdots      &0          &0 					&\cdots  & 1                  
\end{array}
\right)
\label{eqn:bcn1}
\end{equation}
Multiplying (\ref{eqn:cn}) from the \emph{left} by (\ref{eqn:bcn1}) we eliminate the $2n$-column. Similarly, multiplying with elements of $E$ and $V_c$ from the \emph{right} we can eliminate the $2n$-row. 
\begin{equation}
\stackrel{B,V_c}{\leadsto} \left(
\begin{array}{ccc|cccc}
 *   &\cdots      & *         & *   &\cdots  &* & 0        \\
\vdots   &\ddots      &\vdots             &\vdots   &\ddots &   &\vdots           \\ 
 * &\cdots      & *       & * &\cdots &* & 0       \\  \hline
*   &\cdots      & *         & *   &\cdots &* & 0              \\ 
 \vdots   &\ddots      &\vdots             &\vdots   &\ddots &   &\vdots           \\ 
* &\cdots      & *       & * &\cdots &*  & 1                  
\end{array}
\right)
\stackrel{E,V_c}{\leadsto} \left(
\begin{array}{ccc|cccc}
 *   &\cdots      & *         & *   &\cdots  &* & 0        \\
\vdots   &\ddots      &\vdots             &\vdots   &\ddots &   &\vdots           \\ 
 * &\cdots      & *       & * &\cdots &* & 0       \\  \hline
*   &\cdots      & *         & *   &\cdots &* & 0              \\ 
 \vdots   &\ddots      &\vdots             &\vdots   &\ddots &   &\vdots           \\ 
*   &\cdots      & *         & *   &\cdots &* & 0              \\ 
0 &\cdots      & 0       & 0 &\cdots &0  & 1                  
\end{array}
\right)
\label{eqn:cn1}
\end{equation}


In the case that $b_{C_n,C_n}$ is injective, hence bijective, we construct an element in $V_c$:  

\begin{equation}
\left(
\begin{array}{ccc|ccc}
 1       &\cdots      &b_{C_n,H}b_{C_n,C_n}^{-1} &0      		&\cdots  & 0         			\\
\vdots   &\ddots      &\vdots     &\vdots 		   &\ddots  &\vdots          \\ 
0        &\cdots      &b_{C_n,C_n}b_{C_n,C_n}^{-1}         &0      		   &\cdots  & 0       				\\  \hline
0        &\cdots      &0          &1      		   &\cdots  & 0              \\ 
\vdots   &\ddots      &\vdots     &\vdots 		   &\ddots  &\vdots          \\ 
0        &\cdots      &0          &0 				     &\cdots  & 1                 
\end{array}
\right)
\label{eqn:vcn1}
\end{equation}
\noindent
Since the upper left quadrant of ($\ref{eqn:vcn1}$) is the dual of an automorphism that takes values in abelian factors of $G$ it is indeed an element of $V_c$. \\  
Multiplying (\ref{eqn:cn}) from the \emph{left} by (\ref{eqn:vcn1}) we eliminate the upper half of the $2n$-column and  
by multiplying with elements of $E$ and $V_c$ from the \emph{right} we eliminate the $n$-row:   

\begin{equation}
\stackrel{B,V_c}{\leadsto}
\left(
\begin{array}{ccc|ccc}
 *   &\cdots      & *         & *   &\cdots  & 0        \\
\vdots   &\ddots      &\vdots             &\vdots   &\ddots   &\vdots           \\ 
 * &\cdots      & *       & * &\cdots  & b_{C_n,C_n}        \\  \hline
*   &\cdots      & *         & *   &\cdots  & *              \\ 
 \vdots   &\ddots      &\vdots             &\vdots   &\ddots   &\vdots           \\ 
* &\cdots      & *       & * &\cdots  & *                  
\end{array}
\right)
\stackrel{E,V_c}{\leadsto}
\left(
\begin{array}{ccc|ccc}
 *   &\cdots      & *         & *   &\cdots  & 0        \\
\vdots   &\ddots      &\vdots             &\vdots   &\ddots   &\vdots           \\ 
 0 &\cdots      & 0       & 0 &\cdots  & b_{C_n,C_n}       \\  \hline
*   &\cdots      & *         & *   &\cdots  & *              \\ 
 \vdots   &\ddots      &\vdots             &\vdots   &\ddots   &\vdots           \\ 
* &\cdots      & *       & * &\cdots  & *                  
\end{array}
\right)
\label{eqn:cn2}
\end{equation}

\noindent
Case $(2)$: If $v_{C_n,C_m}(c_n)$ is injective it is also bijective, since $|C_m| \leq |C_n|$ by construction. Then we must have $p_n=p_m$ and $k_n=k_m$. Then let $w \in \Aut_c(G)$ be an automorphism such that $w(C_n)=C_m$, $w(C_m)=C_n$ and identity elsewhere. Multiplying with $\left(\begin{smallmatrix} 1 & 0 \\ 0 & w \end{smallmatrix}\right) \in V_c$ from the left returns us to Case (1) when $v_{C_n,C_n}$ is invertible. Similarly, if $b_{C_n,C_m}$ is injective, it has to be bijective because of the same order argument as above. Exchanging the $C_n$ with $C_m$ by an $\left(\begin{smallmatrix} w^* & 0 \\ 0 & 1 \end{smallmatrix}\right) \in V_c$ we return to the Case (1) when $b_{C_n,C_n}$ is invertible. \\ 

\noindent
Case $(3)$: If $v_{C_n,H}$ is injective, we can also assume that $v_{C_n,H}$ is the only injective map in the last column of $\phi$, else we choose Case (1) or Case (2). Since the whole matrix $\phi$ is invertible, there exists an inverse matrix $\phi^{-1} = \left(\begin{smallmatrix} u' & b' \\ a' & v' \end{smallmatrix}\right)$ and then the multiplication of the last right column of $\phi$ with the last upper row of $\phi^{-1}$ has to be $1$. Therefore there has to be a homomorphism $v'_{H,C_n}:H \rightarrow C_n$ such that $w := v'_{H,C_n} \circ v_{C_n,H}:C_n \to C_n$ is injective, and therefore bijective. We have an exact sequence $$0 \to \ker(v'_{H,C_n}) \to H \stackrel{v'_{H,C_n}}{\to} C_n \to 0$$ which splits on the right via $v_{C_n,H} \circ w^{-1}$. Restricting to group-like objects $G(\DG) = \hat{G} \times G$ the row $a_{H,H}a_{C_1,H}...v_{H,H}...v_{C_n,H}$ gives a surjection from $\hat{G}\times G$ to $H$. It maps central elements to central elements because it is surjective and hence the restriction to $C_n$, namely $v_{C_n,H}$, has central image. This implies that $C_n$ is a direct abelian factor of $H$ which is not possible, because $H$ is purely non-abelian per construction. \\

\noindent
Hence we end up with either the form (\ref{eqn:cn2}) or the form (\ref{eqn:cn3}). 
Now we inductively move on to $C_{n-1}$, $C_{n-2}$, ..., $C_1$ where we permute parts with the non-invertible $v's$ to the right lower corner by multiplying with elements of $V_c$. Since $\ker(v)$ has trivial intersection with $H$ per construction, the map $v_{H,H}$ is invertible. As in the Case (1) we can use row and column manipulation to get zeros below and above $v_{H,H}$ as well as left and right. Note that the elements we are constructing are always in $V_c$ because either have abelian image per definition or are restrictions on abelian direct factors of surjections. Only $v_{H,H},u_{H,H}$ do in general not induce $V_c$ elements like that. But corresponding to the automorphism $v_{H,H}^{-1}$ there is a matrix in $V$. Multiplying with this matrix changes the remaining $u_{H,H}$ to $v_{H,H}^*\circ u_{H,H}$ and the $v_{H,H}$ to $\id_H$. \\ 

\noindent
Now we now consider a generator $\chi_n$ of $\hat{C}_n$ and conclude from the fact that $\phi$ is an automorphism the analogous case differentiation from above but now for entries in the remaining $u$ and $a$. With the same arguments as above we move through the columns corresponding to $\hat{C}_{n-1},...,\hat{C}_1,\hat{H}$ and end up with a matrix of the following form: 

\begin{equation}
\left(
\begin{array}{ccc|ccc}
u_{H,H}    &0                      &0                 &0    &0                 &0                     \\
0          &I_k                    &0                 &0    &0                 &0                     \\  
0          &0                      &0                 &0    &0                 &B_m                     \\ \hline
0          &0                      &0                 &1    &0                 &0                     \\
0          &0                      &0                 &0    &I_k               &0                     \\  
0          &0                      &A_m               &0    &0                 &V_m                    \\  
\end{array}
\right)
\label{eqn:cn3}
\end{equation}

\noindent
where $k+m+1=n$, $B_m$, $A_m$ are diagonal $m\times m$-matrices with isomorphisms on the diagonal, $I_k$ an $k\times k$-identity matrix and $V_m$ a $m \times m$-matrix with non-invertible homomorphisms as entries. Further, since $H$ is purely non-abelian and by Lemma \ref{lm_KeilbergTechnical} $\ker(u)$ is contained in an abelian direct factor we deduce that $u_{H,H}$ is an isomorphism. Also by Lemma \ref{lm_KeilbergTechnical} we know that the composition 

$$
\begin{pmatrix} u_{H,H}^* &0  &0 \\
0          &I_k  &0 \\  
0          &0    &0  \end{pmatrix} 
\begin{pmatrix} 1 &0  &0 \\
0          &I_k  &0 \\  
0          &0    &V_m  \end{pmatrix} 
= \begin{pmatrix} u_{H,H}^* &0  &0 \\
0          &I_k  &0 \\  
0          &0    &0  \end{pmatrix} 
$$

\noindent
has to be a normal homomorphism, hence $u_{H,H}$ has to be a central automorphism. Therefore we get (\ref{eqn:cn3}) with $u_{H,H}=1$ by multiplying with the inverse in $V_c$. Our final step is normalizing the $A_m$ by multiplying with an element in $V_c$ corresponding to 

$$\begin{pmatrix} 1 &0  &0 \\
0          &I_k  &0 \\  
0          &0    &B_m^{-1}A_m^{-1}  \end{pmatrix} 
$$

Hence we end up with a twisted reflection:  

\begin{equation}
\left(
\begin{array}{ccc|ccc}
1          &0                      &0                 &0    &0                 &0                     \\
0          &I_k                    &0                 &0    &0                 &0                     \\  
0          &0                      &0                 &0    &0                 &B_m                   \\ \hline
0          &0                      &0                 &1    &0                 &0                     \\
0          &0                      &0                 &0    &I_k               &0                     \\  
0          &0                      &B_m^{-1}          &0    &0                 &V_m                   \\  
\end{array}
\right)
\label{eqn:cn4}
\end{equation} 

\noindent
(iii) If $C \cong C'$ it follows from Proposition \ref{reflections} (ii) that $r_{H,C,\delta}$ and $r_{H',C',\delta'}$ are conjugate to each other. This implies that their double cosets have non trivial intersection and hence are equal.  \\ 
Assume that the two double cosets corresponding to $r'_{(H',C',\delta')}$ and $r_{(H,C,\delta)}$ are equal. Then there are $w,w',v,v' \in \Aut(G)$, $a \in \Hom(\widehat{Z(G)},Z(G))$, $b' \in \Hom(G,\widehat{G})$ such that 
\begin{equation*} 
\begin{pmatrix} w^* & a \\ 
                0 & v 
\end{pmatrix}  
\begin{pmatrix} \hat{p}_H & \delta \\ 
                \delta_C^{-1} & p_H 
\end{pmatrix}
= 
\begin{pmatrix} \hat{p}_{H'} & \delta' \\ 
                \delta^{'-1} & p_{H'} 
\end{pmatrix}
\begin{pmatrix} w^{'*} & 0 \\ 
                b' & v' 
\end{pmatrix} 
\end{equation*}
This implies in particular $v \circ p_H = p_{H'} \circ v'$. Then $C = \ker(p_H) = \ker(v^{-1} \circ p_{H'} \circ v' ) = v^{'-1}(\ker(p_{H'})) = v^{'-1}(C')$. Hence $v'$ restricted to $C$ gives the isomorphism $C \simeq C'$. \\  

\noindent
(iv) Let $\phi$ be a general element in $\Aut_{Hopf}(\DG)$ as in (\ref{eqn:cn1}). We use the same arguments as in (ii) to get to the case differentiation for every column. We can produce zeros in each row by multiplying $\phi$ with $E$ and $V_c$ form the right, except when we have invertible entries in $a$ and $u$. In this case multiplying with $E,V_c$ from the right can only produce zeros in $u$ and $a$ respectively. Hence we end up with: 
\begin{equation}
\left(
\begin{array}{ccc|ccc}
1          &0                      &0                 &*    &*                 &*                     \\
0          &I_k                    &0                 &*    &*                 &*                     \\  
0          &0                      &0                 &0    &0                 &B_m                   \\ \hline
0          &0                      &0                 &1    &0                 &0                     \\
0          &0                      &0                 &0    &I_k               &0                     \\  
0          &0                      &B_m^{-1}          &*    &*                 &*                     \\  
\end{array}
\right)
\label{eqn:cn5}
\end{equation}
where $B_m$ is a diagonal $m \times m$-matrix with isomorphisms on the diagonal, $I_k$ an $k \times k$-identity matrix and $k+m=n+1$.    
Multiplying again from the right but now with elements of $B$ (which was not allowed in (ii)) and elements of $V_c$ we eliminate the $*$ which results in a (non-twisted) reflection. \\

\noindent
(v) This is similar to above. We start with (\ref{eqn:cn1}) an move from column to column as in (ii) and (iv). We identify the invertible entries and can clean up each \emph{column} by multiplying with elements of $B$ and $V_c$ from the \emph{left}. Finally as in $(ii)$ we get a non-invertible twist $V_m$ below $B_m$ but contrary to $(ii)$ we can multiply with $E$ from the \emph{left} and eliminate the $V_m$. Therefore we again end up with a (non-twisted) reflection. This concludes the proof of Theorem \ref{thm_cell}.
\end{proof}

\section{On the decomposition of $\H^2_L(\DG^*)$}\label{seclazy}

In the section above we have decomposed $\Aut_{Hopf}(\DG^*)$ into manageable subgroups. Now we would like to decompose $\H^2_L(\DG^*)$ in analogy to the K\"unneth formula. First note that the lazy cohomology of $\DG$ is the same as the lazy cohomology of the tensor product $k^G \otimes kG$ (see e.g \cite{Bich} Corollary 4.11). This can be seen from the fact that $\DG$ is the Doi-twist of the tensor product. $\DG^*$ on the other hand is the Drinfeld twist of $kG \otimes k^G$, hence has a twisted coalgebra structure. Let us consider the exact sequence of Hopf algebras
\begin{align}\label{DGSeq} kG \stackrel{p}{\longleftarrow} \DG^* \stackrel{\iota}{\longleftarrow} k^G \end{align}
\noindent
with the obvious splitting $s:\DG^* \to k^G$ and the algebra map $t:kG \to \DG^*$. Let us collect the following property: 

\begin{lemma}\label{DGlazy}
Let $\sigma \in \Z^2(\DG^*)$ then $\sigma$ is lazy if and only if for all $g,h,x,y \in G$:  
\begin{itemize}
\item If $gh = g^xh^y$ we have $\sigma(g \times e_{x}, h \times e_{y}) = \sigma(g^{t} \times e_{x^t}, h^{t} \times e_{y^t})$ for all $t \in G$
\item If $gh \neq g^xh^y$ we have $\sigma(g \times e_{x}, h \times e_{y}) = 0$
\end{itemize}
Further, $\eta \in \Reg^1(\DG^*)$ is lazy if and only if for all $g,x \in G$:
\begin{itemize}
\item If $gx=xg$ then $\eta(g^{t} \times e_{x^t}) =  \eta(g \times e_{x})$
\item If $gx \neq xg$ we have $\eta(g \times e_{x})=0$
\end{itemize}
\end{lemma}
\begin{proof}
The lazy condition is $\sigma * \mu = \mu * \sigma $. The left hand side is:
\begin{align*}
\sum_{x_1x_2 = x, y_1y_2 = y}\sigma(g \times e_{x_1}, h \times e_{y_1}) (g^{x_1}h^{y_1} \times e_{x_2}e_{y_2}) &= \sum_{t}\sigma(g \times e_{xt^{-1}}, h \times e_{yt^{-1}}) (g^{xt^{-1}}h^{yt^{-1}} \times e_{t}) 
\end{align*} 
the right hand side is: 
\begin{align*}
\sum_{x_1x_2 = x, y_1y_2 = y}\sigma(g^{x_1} \times e_{x_2}, h^{y_1} \times e_{y_2}) (gh \times e_{x_1}e_{y_1}) &= \sum_{t}\sigma(g^{t} \times e_{t^{-1}x}, h^{t} \times e_{t^{-1}y}) (gh \times e_{t}) 
\end{align*}
Equivalently, for all $t,z,g,h,x,y \in G$: 
\begin{align*}
\sigma(g \times e_{x}, h \times e_{y}) \delta_{z,g^{x}h^{y}} = \sigma(g^{t} \times e_{x^t}, h^{t} \times e_{y^t}) \delta_{z,gh} 
\end{align*}
Similarly an $\eta \in \Reg^1(\DG^*)$ is lazy iff: $\eta * \id = \id * \eta$. Then comparing the left hand side:
\begin{align*}
\sum_{x_1x_2=x} \eta(g \times e_{x_1})(g^{x_1} \times e_{x_2}) = \sum_{t} \eta(g \times e_{xt^{-1}})(g^{xt^{-1}} \times e_{t}) 
\end{align*} 
with the right hand side: 
\begin{align*}
\sum_{y_1y_2 = x}\eta(g^{y_1} \times e_{y_2})(g \times e_{y_1}) = \sum_{t}\eta(g^{t} \times e_{t^{-1}x})(g \times e_{t})
\end{align*}
leads to $\eta(g^{t} \times e_{x^t})\delta_{g,z} =  \eta(g \times e_{x})\delta_{g^{x},z}$. 
\end{proof}

It is natural to expect that $\H^2_L(\DG^*)$ decomposes in some way into parts that depend on $kG$, $k^G$ and some sort of pairings between both. We start with the $kG$ part:

\begin{definition}
Let $\Z^2_{inv}(G,k^\times)$ be those $2$-cocycles $\beta\in \Z^2(G,k^\times)$ that fulfill $$\beta(g,h)=\beta(g^t,h^t) \qquad \forall t \in G$$ and let $\C^1_{inv}(G,k^\times)$ be maps $\eta:G \to k^\times$ such that $\eta(g)=\eta(g^t) \; \forall t \in G$. We define the cohomology group of conjugation invariant cocycles to be the group
\begin{align}\H^2_{inv}(G,k^\times):=\Z^2_{inv}(G,k^\times)/\d(\C^1_{inv}(G,k^\times)) \end{align}
\end{definition}~

\begin{lemma} There is a set theoretic map $\H^2_L(\DG^*) \rightarrow \H^2_{inv}(G);\sigma \mapsto \beta_{\sigma}$ defined by $$\beta_{\sigma}(g,h)=\sigma(g \times 1, h \times 1)$$ Further, it is a split of $\H^2_{inv}(G) \to \H^2_L(\DG^*)$ assigning a conjugation invariant $2$-cocycle $\beta$ on $G$ to $\sigma_{\beta}$ defined by $\sigma_{\beta}(g \times e_x, h \times e_y) = \beta(g,h)\epsilon(e_x)\epsilon(e_y)$. \\
\end{lemma}
\begin{proof}
We apply the cocycle condition ($\ref{cocycle}$) for $a = g \times 1$, $b = h \times 1$ and $c = r \times 1$ for $g,h,r \in G$:  
\begin{align}
\sum_{t,s}\sigma(g \times e_t,h \times e_s)\sigma(g^th^s \times 1, r \times 1) &=  \sum_{gh = g^sh^t }\sigma(g \times e_t,h \times e_s)\sigma(gh \times 1, r \times 1) \\ &=  \beta_{\sigma}(g,h)\beta_{\sigma}(gh,r) \\
&=\sum_{s,z}\sigma(h \times e_s,r \times e_z)\sigma(g \times 1,h^sr^z \times 1) \\ &=\beta_{\sigma}(h,r)\beta_{\sigma}(g,hr)
\end{align}
where we have used Lemma \ref{DGlazy}. This shows that $\beta_{\sigma}$ is a $2$-cocycle with coefficients $k$. Now we argue that $\beta_{\sigma}(g,h) \neq 0$ for all $g,h \in G$. This is not clear because $g \times 1$ is not a group-like element in $\DG^*$. Let $\sigma^{-1}$ be the convolution inverse of $\sigma$ first we claim that $\beta_{\sigma}(g,g^{-1}) \neq 0$. Here we use again Lemma \ref{DGlazy}: 
\begin{align*}
1 &= \sigma*\sigma^{-1}(g \times 1, g^{-1} \times 1) = \sum_{x,y \in G}\sigma(g \times e_x, g^{-1} \times e_y)\sigma^{-1}(g^x \times 1, (g^{-1})^y \times 1) \\
   &= \sum_{\myover{x,y \in G}{1 = g^x(g^{-1})^y}}\sigma(g \times e_x, g^{-1} \times e_y)\sigma^{-1}(g^x \times 1, (g^{-1})^x \times 1) \\
	 &= \sum_{x,y \in G}\sigma(g \times e_x, g^{-1} \times e_y)\sigma^{-1}(g \times 1, (g^{-1}) \times 1) \\
	&=\sigma(g \times 1, g^{-1} \times 1)\sigma^{-1}(g \times 1, (g^{-1}) \times 1)
\end{align*}  
Using the cocycle condition: $\beta_{\sigma}(g,g^{-1}) = \beta_{\sigma}(g,g^{-1})\beta_{\sigma}(gg^{-1},h) = \beta_{\sigma}(g^{-1},h)\beta_{\sigma}(g,g^{-1}h)$ since $\beta(g,g^{-1})$ is invertible we also have that $\sigma(g \times 1,h \times 1)=\beta_{\sigma}(g,h) \neq 0$ for all $g,h \in G$. Since $\sigma$ is conjugation invariant so is $\beta_{\sigma}$. \\
If $\beta$ is a $2$-cocycle then it is obvious that $\sigma_{\beta}$ is a $2$-cocycle and it is lazy since it fulfills the conditions in Lemma $\ref{DGlazy}$. Also, the second part of Lemma \ref{DGlazy} implies that a lazy coboundary in $\DG^*$ is mapped to a conjugation invariant coboundary on $G$. \\   
\end{proof}

\begin{lemma}
There is a natural map $\H^2_L(\DG^*) \to \H^2_L(k^G);\sigma \mapsto \alpha_{\sigma}$ defined by $$\alpha_{\sigma}(e_x,e_y)=\sigma(1 \times e_x,1 \times e_y)$$ Further, $\alpha$ restricted to the character group $\widehat{G}$ gives a $2$-cocycle in $\H^2(\widehat{G},k^\times)$. 
\end{lemma}
\begin{proof}
Since the map $\iota$ in sequence (\ref{DGSeq}) is Hopf it follows that $\alpha_{\sigma} \in \Z^2(k^G)$. Comparing Example \ref{exm_LazyIff} with Lemma \ref{DGlazy} its easy to see that $\alpha{\sigma} \in \Z^2_L(\DG^*)$. Also, elements $1 \times \chi$ for $\chi \in \hat{G}$ are group-like elements in $\DG^*$. Therefore the convolution invertibility of $\sigma$ implies that $\alpha_{\sigma}$ is indeed a $2$-cocycle in $\Z^2(\hat{G},k^{\times})$. \\      
\end{proof}

\begin{definition}
The group of \emph{lazy bialgebra pairings} $\P_L(kG,k^G)$ consists of convolution invertible $k$-linear maps $\lambda:kG \otimes k^G \rightarrow k$ such that for all $x,t \in G$, $f \in k^G$: 
\begin{equation*}
\begin{split}
\lambda(gh,f) = \lambda(g,f_{1})\lambda(h,f_{2}) &\hspace{0.5cm} \lambda(g,f*f') = \lambda(g,f)\lambda(g,f') \hspace{0.5cm} \lambda(g,e_x) = \lambda(tgt^{-1},e_{txt^{-1}})  \\
&\lambda(1,f) = \epsilon_{k^G}(f) \qquad \lambda(g,1) = \epsilon_{kG}(g) 
\end{split}
\end{equation*}
\end{definition}

\medskip

Bialgebra pairings $\P(kG,k^G)$ are in bijection with the group $\Hom(G,G)$, similarly lazy bialgebra pairings $\P_L(kG,k^G)$ are in bijection with group homomorphisms $f \in \Hom(G,G)$ that are conjugation invariant $f(g)= f(g^t) \forall g,t \in G$.  

\begin{lemma} There is a group homomorphism $\pi: \H^2_L(\DG^*) \rightarrow \P_L(kG,k^G); \sigma \mapsto \lambda_\sigma$ defined by
\begin{align}\label{lambdasigma}
\lambda_\sigma(g,f) &= \sigma^{-1}((1 \times f)_1,(g \times 1)_1)\sigma((g\times 1)_2,(1 \times f)_2) \\
										&= \sum_{t,x,y \in G}f(xy)\sigma^{-1}(1 \times e_{x}, g \times e_t)\sigma(t^{-1}gt \times 1,1 \times e_y)
\end{align} 
\end{lemma}
\begin{proof} Recall the coproduct in $\DG^*$: $\Delta(g \times 1) = \sum_{t}(g \times e_t) \otimes (t^{-1}gt\times 1)$. \\ 
We check that $\pi$ is a well-defined group homomorphism. It is more convenient to be slightly more general here. \\ 

\noindent
Let $a,b,c \in H$ for a Hopf algebra $H$ and $\sigma \in \Z^2_L(H)$. Consider the following map $$ \tau:H\times H \to k; \quad \tau(a,b):=\sigma^{-1}(b_1,a_1)\sigma(a_2,b_2)$$ The laziness condition $\sigma(b_1,c_1)b_2c_2 = b_1c_1 \sigma(b_2,c_2)$ implies that we can commute terms like as the following 
\begin{align}\label{proofeq1}
\sigma(b_1,c_1)\sigma(a,b_2c_2) &= \sigma(a,b_1c_1)\sigma(b_2,c_2) 
\end{align} 

\noindent
Using the above formula and the $2$-cocycle condition two times 
\begin{align} 
\sigma(c,ab)&=\sigma^{-1}(a_1,b_1)\sigma(c_1,a_2)\sigma(c_2a_3,b_2) \label{proofcs1} \\
\sigma^{-1}(ab,c)&=\sigma^{-1}(a_1,b_1c_1)\sigma^{-1}(b_2,c_2)\sigma(a_2,b_3) \label{proofcs2}
\end{align}
Now we can show that $\tau$ is multiplicative in the first argument: 
\begin{align*}
\tau(a_1,c_1)\tau(b_1,c_2)&\tau^{-1}(a_2b_2,c_3) = \sigma^{-1}(c_1,a_1)\sigma(a_2,c_2)\sigma^{-1}(c_3,b_1)\sigma(b_2,c_4)\sigma^{-1}(a_3b_3,c_5)\sigma(c_6,a_4b_4) \\
&\stackrel{(\ref{proofcs1}),(\ref{proofcs2})}{=} \sigma^{-1}(c_1,a_1)\sigma(a_2,c_2)\sigma^{-1}(c_3,b_1)\sigma(b_2,c_4) \\ & \qquad\qquad \sigma^{-1}(a_3,b_3c_5)\sigma^{-1}(b_4,c_6)\sigma(a_4,b_5)\sigma^{-1}(a_5,b_6)\sigma(c_7,a_6)\sigma(c_8a_7,b_7) \\
&\stackrel{(\ref{proofeq1})}{=} \sigma^{-1}(c_1,a_1)\sigma(a_2,c_2)\sigma^{-1}(c_3,b_1)\sigma^{-1}(a_3,b_2c_4)\sigma(c_5a_4,b_3)\sigma(c_6,a_5)  \\
\end{align*}    

If all coproduct terms $a_i$ and $c_j$ commute and if all the $b_i$ and $c_l$ commute we go on:  
\begin{align*}
&= \sigma^{-1}(c_1,a_1)\underbrace{\sigma(a_2,c_2)\sigma^{-1}(c_3,b_1)\sigma^{-1}(a_3,c_4b_2)\sigma(a_4c_5,b_3)}_{\epsilon}\sigma^{-1}(c_6,a_5) = \epsilon(abc) 
\end{align*}
Where we again used the cocycle condition in the middle. Note that $\tau(a,b) = \tau^{-1}(b,a)$ hence the above shown property of $\tau$ implies: 
$$ \tau(c,ab) = \tau(c_1,b)\tau(c_2,a)$$ Taking $H=\DG^*$, $\sigma \in Z_L^2(\DG^*)$, we get $\tau(g \times 1,1 \times e_x) = \lambda_\sigma(g,e_x)$ for all $g,x \in G$ (compare with equation (\ref{lambdasigma})). Note that the coproduct terms of $(1 \times e_x)$ and $(1 \times e_y)$ in $\DG^*$ stay in $1 \times k^G$ and therefore the coproduct terms of $(g \times 1)$ commute with all coproduct terms $(1 \times e_x)$ and $(1 \times e_y)$. Using this we show that $\lambda_\sigma$ is multiplicative:  

\begin{align*}
\lambda_\sigma(g,e_x*e_y) &= \tau(g \times 1, (1 \times e_x)(1 \times e_y)) = \tau((g\times 1)_2 , 1 \times e_{x})\tau((g\times 1)_1 , 1 \times e_y) \\
                          &= \sum_{t \in G} \tau(g \times e_t, 1 \times e_{y})\tau(t^{-1}gt \times 1, 1 \times e_x) \\
													&= \tau(g \times 1, 1 \times e_{y})\tau(g \times 1, 1 \times e_x) \\
													&= \lambda_{\sigma}(g,e_x)\lambda_{\sigma}(g,e_y)
\end{align*}
Similarly: 
\begin{align*}
\lambda_\sigma(gh,e_x) &= \tau((g \times 1)(h\times 1), 1 \times e_x) = \sum_{x_1x_2=x} \tau(g \times 1, 1 \times e_{x_1})\tau(h \times 1, 1 \times e_{x_2}) \\
&= \sum_{x_1x_2=x}\lambda_{\sigma}(g,e_{x_1})\lambda_{\sigma}(h,e_{x_2})
\end{align*}

The map also induces a well-defined map on cohomology, since for any $a,b \in H$ 
\begin{align*} 
(\d\mu)^{-1}(b_1,a_1)\d\mu(a_2,b_2) = \mu(b_1a_1)\mu^{-1}(b_2)\mu^{-1}(a_2)\mu(a_3)\mu(b_3)\mu^{-1}(a_4b_4) = \mu(b_1a_1)\mu^{-1}(a_2b_2)												
\end{align*}
and therefore $(\d\mu)^{-1}(b_1,a_1)\d\mu(a_2,b_2) = \epsilon(ab)$ if $a,b$ commute. \\
\end{proof} 


\begin{proposition}~
\begin{itemize}
\item Let $\P_{c}(kG,k^G) \subset \P_L(kG,k^G)$ be the subgroup of central lazy bialgebra pairings. These are $\lambda \in \P_L(kG,k^G)$ such that for $g \in G$: $\lambda(g,e_x)=0$ if $x$ not in $Z(G)$. Then there is a natural group homomorphism 
\begin{equation}\label{embpair}
\P_c(kG,k^G) \rightarrow \H^2_L(\DG^*); \lambda \mapsto \sigma_\lambda = \lambda \circ (p \otimes s) 
\end{equation}
where $p,s$ were defined in the splitting sequence $(\ref{DGSeq})$. Note that $\P_c(kG,k^G)$ is in bijection with $\Hom(G,Z(G))$. 
\item Let $\Z_c^2(k^G) \subset \Z^2_L(k^G)$ be the subgroup of central lazy $2$-cocycles. These are $\alpha \in \Z^2_L(k^G)$: $\alpha(e_x,e_y)=0$ if $x$ or $y$ not in $Z(G)$. We define $\H^2_c(k^G)$ to be the quotient by central coboundaries $\d\eta \in \Z_c^2(k^G)$ for $\eta \in \Reg^1(k^G)$. Then there is a natural group homomorphism $\H^2_c(k^G) \to \H^2_L(\DG^*); \alpha \mapsto \sigma_{\alpha}$ defined by
\begin{align*}
\sigma_{\alpha}(g \times e_x,h \times e_y) = \alpha(e_x,e_y) \epsilon(g)\epsilon(h)
\end{align*}  
\end{itemize}
\end{proposition}
\begin{proof}
The cocycle condition (equation (\ref{cocycle})) for $\sigma_{\lambda}$ reduces to
$$ \lambda(x,e_y)\lambda(xy^{-1}zy,e_{w}) = \sum_{w_1w_2=w}\lambda(z,e_{w_1})\lambda(x,e_y*e_{w_2})$$ for all $x,y,z,w \in G$. Using the properties of the pairing $\lambda$ we check this equality:
\begin{align*}
\lambda(x,e_y)\lambda(xy^{-1}zy,e_{w}) &=  \sum_{w_1w_2=w}\lambda(x,e_y)\lambda(x,e_{w_1})\lambda(y^{-1}zy,e_{w_2}) \\ 
&= \sum_{w_1w_2=w}\lambda(x,e_y*e_{w_1})\lambda(y^{-1}zy,e_{w_2}) \\
&= \lambda(x,e_{y})\lambda(y^{-1}zy,e_{y^{-1}w}) = \lambda(x,e_{y})\lambda(z,e_{wy^{-1}}) \\
&=\sum_{w_1w_2=w} \lambda(z,e_{w_1})\lambda(x,e_y*e_{w_2})
\end{align*} 
The fact that $\lambda(g,e_x)$ is conjugation invariant and zero if $x$ is not central ensures that $\sigma_{\lambda}$ is lazy in $\Z^2(\DG^*)$ (see again Lemma \ref{DGlazy}). The fact that $\lambda \mapsto \sigma_{\lambda}$ is a homomorphism is straightforward to check. Similarly, since $\alpha$ is lazy on $k^G$ this implies conjugation invariance and since $\alpha$ is central Lemma \ref{DGlazy} implies that $\sigma_{\alpha}$ is lazy on $\DG^*$.  \\
\end{proof}

\begin{conjecture}
The group $\H^2_L(\DG^*)$ is generated by $\H^2_c(k^G)$, $\P_c(kG,k^G)$ and $\H_{inv}^2(G,k^\times)$. Further, there is an exact factorization: $\H^2_L(\DG^*) = \H^2_c(k^G) \P_c(kG,k^G) \H_{inv}^2(G,k^\times)$. \\
\end{conjecture}

\begin{conjecture}
The set $\H^2(\DG^*)$ is in bijection with $\H^2(k^G) \times \P(kG,k^G) \times \H^2(G,k^\times)$. \\   
\end{conjecture}

\begin{lemma}\label{lm_inKernelExact}
A lazy $2$-cocycle $\sigma$ such that $\beta_{\sigma}$ cohomologically trivial in $\H^2(G,k^\times)$, $\alpha_{\sigma}$ cohomologically trivial in $\Z^2(k^G)$ and $\lambda(g,e_x)=\epsilon(e_x)$ fulfills the property that it is cohomologically trivial in $\H^2(DG^*)$ but not necessarily in $\H^2_L(\DG^*)$. \\
\end{lemma}
\begin{proof}
Let $\beta_{\sigma}= \d\nu$ and $\alpha_{\sigma}= \d \eta$. These to do not induce lazy coboundaries on $\DG^*$. But apply the $2$-cocycle condition several times on $\sigma$ by separating the $kG$ and $k^G$ parts:    
\begin{align} 
\sigma&(g \times e_x, h \times e_y) = \sigma((g \times 1)(1 \times e_x), h \times e_y) \nonumber \\
&= \sum_{\myover{x_1x_2x_3=x}{y_1y_2=y}} \sigma^{-1}(g \times 1 ,1 \times e_{x_1})\sigma(1 \times e_{x_2}, (h \times 1) (1 \times e_{y_1}))\sigma(g \times 1, (h \times 1 )(1 \times e_{x_3}e_{y_2 })) \nonumber \\
&= \sum_{\myover{x_1x_2x_3x_4x_5=x}{y_1y_2y_3y_4y_5=y}} \sigma^{-1}(g \times 1, 1 \times e_{x_1} )\sigma^{-1}(1 \times e_{y_1}, h \times 1) \sigma(1 \times e_{x_2}, 1 \times e_{y_2}) \nonumber \\ & \qquad \sigma(1 \times e_{x_3}e_{y_3}, h \times 1)\sigma^{-1}(h \times 1, 1 \times e_{x_4}e_{y_4})\sigma(g \times 1, h \times 1)\sigma(gh \times 1, 1 \times e_{x_5}e_{y_5}) \nonumber \\
&=\sum_{\myover{x_1x_2=x}{y_1y_2=y}} \sigma^{-1}(g \times 1, 1 \times e_{x_1} )\sigma^{-1}(1 \times e_{y_1}, h \times 1)\d\nu(g,h)d\eta(e_{x_2},e_{y_2})\sigma(gh \times 1, 1 \times e_{x_2}e_{y_2}) \label{con2}
\end{align}
Now let $\mu(g \times e_x) := \sigma^{-1}(g \times 1, 1 \times e_x)$ and check that together with $\eta$ this gives us the desired almost lazy coboundary:   
\begin{align}
&\d(\mu*(\eta \otimes \nu))(g \times e_x, h \times e_y) \nonumber \\ &= \sum_{\myover{x_1x_2 =x}{y_1y_2=y}}\mu*(\eta \otimes \nu)(g \times e_{x_1})\mu*(\eta \otimes \nu)(h \times e_{y_1})\mu*(\eta \otimes \nu)(g^{x_1}h^{y_1} \times e_{x_2}e_{y_2}) \nonumber \\ 
&= \sum_{\myover{x_1x_2x_3x_4 =x}{y_1y_2y_3y_4=y}} \sigma^{-1}(g \times 1, 1 \times e_{x_1})\nu(g)\eta(e_{x_2})\sigma^{-1}(h \times 1, 1 \times e_{y_1})\nu(h)\eta(e_{y_2}) \nonumber \\ & \qquad \sigma(g^{x_1x_2}h^{y_1y_2} \times 1, 1 \times e_{x_3}e_{y_3})\nu(gh)\eta(e_{x_4}e_{y_4}) \nonumber \\
&= \sum_{\myover{x_1x_2t=x}{y_1y_2t=y}} \sigma^{-1}(g \times 1, 1 \times e_{x_1})\sigma^{-1}(h \times 1, 1 \times e_{y_1})\d\nu(g,h)\d\eta(e_{x_2},e_{y_2})\sigma(g^{xt^{-1}}h^{yt^{-1}} \times 1, 1 \times e_t) \nonumber \\ 
&= \sum_{\myover{x_1x_2t=x}{y_1y_2t=y}} \sigma^{-1}(g \times 1, 1 \times e_{x_1})\sigma^{-1}(h \times 1, 1 \times e_{y_1})\d\nu(g,h)d\eta(e_{x_2},e_{y_2})\sigma(g^{x}h^{y} \times 1, 1 \times e_t) \label{con1}
\end{align}  
Here we used the lazy property of $\sigma$ as in Lemma \ref{DGlazy} but also in order to commute with the $\eta$ factos as follows: $\sigma(a_1,b_1)\eta(a_2b_2) = \eta(a_1b_1)\sigma(a_2b_2)$. Comparing ($\ref{con1}$) with ($\ref{con2}$) proves the statement. 
\end{proof}

\noindent
At the end we want to present an interesting special case of the above decomposition.    
Let us call a $2$-cocycle $\sigma \in \Z_L^2(\DG^*)$ symmetric if for all $g,t,h,s \in G$: 

\begin{align*} 
\sigma(g\times e_t,h^g\times e_s)
=\sigma(h\times e_{gs(g^{-1})^t},g\times e_t) 
\end{align*}

This is motivated by the study of lazy \emph{braided} monoidal autoequivalences of $\DG\md\mod$, where such a $2$-cocycle defines an autoequivalence of $\DG\md\mod$ that is identity on objects and morphisms but has a non-trivial monoidal structure determined by $\sigma$.  \\  

\begin{lemma}\label{lm_symmetric}
A symmetric lazy $2$-cocycle on $\DG^*$ is cohomologically equivalent in $\H^2(DG^*)$ to a lazy $2$-cocycle in the image of map $\Z^2_{inv}(G) \to \Z_L^2(\DG^*)$.
\end{lemma}
\begin{proof}
From the symmetry condition follows that $\beta_{\sigma}(g,h^g)=\beta_{\sigma}(h,g)$, that $\alpha_{\sigma}$ is symmetric and that $\lambda_{\sigma}=1$. Multiplying $\sigma$ from the left by $\sigma_{\beta_{\sigma}}^{-1}$ gives a $2$-cocycle $\sigma'$ that fulfills all the properties in Lemma \ref{lm_inKernelExact}, in particular $\beta_{\sigma'}$ cohomologically trivial in $\H^2(G,k^\times)$, hence $\sigma'$ is cohomologically trivial in $\H^2(\DG^*)$.  
\end{proof}

\noindent{\sc Acknowledgments}: 
We are grateful to C. Schweigert for many helpful
discussions. We also want to thank M. Keilberg for comments and suggestions. The authors are partially supported by the DFG Priority Program SPP 1388 ``Representation Theory'' and the Research Training Group 1670 ``Mathematics Inspired by String Theory and QFT''. S.L. is currently on a research stay supported by DAAD PRIME, funded by BMBF and EU Marie Curie Action.

\end{document}